\documentclass[a4paper,11pt]{article}
\usepackage[utf8]{inputenc}
\usepackage[T1]{fontenc}
\usepackage{amsthm}
\usepackage{amsmath}
\usepackage{amssymb}
\usepackage{mathtools}
\usepackage{enumitem}
\usepackage{multirow}
\usepackage[normalem]{ulem}
\usepackage{tikz}
\usepackage{tikz-cd}
\usetikzlibrary{arrows,angles,arrows.meta, matrix}
\usepackage[hidelinks]{hyperref}
\usepackage{xcolor}

\newcommand{\mb}[1]{\ensuremath{\mathbb{#1}}}

\newcommand{\R}{\mb{R}}

\newfont{\bl}{msbm10 scaled \magstep2}

\newcommand{\beq}{\begin{equation}}
\newcommand{\eeq}{\end{equation}}
\newcommand{\notmid}{\mid\kern-0.5em\not\kern0.5em}

\newcommand{\eps}{\varepsilon}

\newcommand{\vphi}{\varphi}

\newtheorem{thm}{Theorem}[section]
\newtheorem{lem}[thm]{Lemma}

\newtheorem{cor}[thm]{Corollary}
\newtheorem{rem}[thm]{Remark}
\newtheorem{defi}[thm]{Definition}

\newcommand{\diamfin}{\mathrm{diam_{fin}}}

\newcommand{\LpLS}{Lorentzian pre-length space }
\newcommand{\LpLSn}{Lorentzian pre-length space}

\newcommand{\bx}{\bar{x}}
\newcommand{\by}{\bar{y}}
\newcommand{\bz}{\bar{z}}
\newcommand{\bp}{\bar{p}}
\newcommand{\bq}{\bar{q}}
\newcommand{\btau}{\bar{\tau}}

\newcommand{\hx}{\hat{x}}
\newcommand{\hy}{\hat{y}}
\newcommand{\hz}{\hat{z}}

\newcommand{\tx}{\tilde{x}}

\newcommand{\balp}{\bar{\alpha}}
\newcommand{\bbet}{\bar{\beta}}
\newcommand{\bgam}{\bar{\gamma}}

\newtheorem{pop}[thm]{Proposition}

\theoremstyle{definition} % just in case the style had changed

\newcommand{\thistheoremnam}{}
\newtheorem*{genericthm*}{\thistheoremnam}
\newenvironment{chapt*}[1]
  {\renewcommand{\thistheoremnam}{#1}%
   \begin{genericthm*}}
  {\end{genericthm*}}

\renewcommand{\labelenumi}{(\roman{enumi})}

\newcommand{\lm}[1]{\mathbb{L}^2(#1)}
\newcommand{\ma}{\measuredangle}

\newcommand{\arcosh}{\mathrm{arcosh}}

\usepackage{accents}

\DeclareMathOperator{\md}{md}
\DeclareMathOperator{\sn}{sn}
\DeclareMathOperator{\cn}{cn}

\renewcommand{\labelenumi}{(\roman{enumi})}
\renewcommand\theenumi\labelenumi

\title{On curvature bounds in Lorentzian length spaces}
\author{Tobias Beran\thanks{{\tt tobias.beran@univie.ac.at}, Faculty of Mathematics, University of Vienna, Austria.}\ , Michael Kunzinger\thanks{{\tt michael.kunzinger@univie.ac.at},  Faculty of Mathematics, University of Vienna, Austria.}\ and Felix Rott\thanks{{\tt felix.rott@univie.ac.at}, Faculty of Mathematics, University of Vienna, Austria.}}
\date{}

\begin{document}
\maketitle

\begin{abstract}
We introduce several new notions of (sectional) curvature bounds for Lorentzian pre-length spaces: On the one hand, we provide convexity/concavity conditions for the (modified) time separation function, and, on the other hand, we study four-point conditions, which are suitable also for the non-intrinsic setting. 
Via these concepts we are able to establish (under mild assumptions) the equivalence of all previously known formulations of curvature bounds. 
In particular, we obtain the equivalence of causal and timelike curvature bounds as introduced in \cite{KS18}. 
\bigskip

\medskip


\end{abstract}
\tableofcontents
\section{Introduction}

The theory of Lorentzian length spaces aims to give a synthetic description of Lorentzian geometry. 
Inspired by the transformative effect its metric predecessor (cf., e.g.,  \cite{BBI01,BH99,AKP19}) has had on the field of Riemannian geometry, after its introduction in \cite{KS18} the theory has quickly branched out from Lorentzian Alexandrov geometry (e.g., \cite{AGKS19,BS22,BORS23})
into a variety of fields, in particular into Optimal Transport and Metric Measure Geometry (e.g., \cite{CM20,McS22,B23}),  
causality theory (e.g., \cite{ACS20,BGH21,KS21}), and General Relativity (e.g., \cite{GKS19,MS23,McC23}). 

During the initial development of the theory, one of the main goals was to establish a synthetic version of sectional curvature bounds via triangle comparison, a characterization which in the smooth setting is known even for semi-Riemannian manifolds due to \cite{AB08}. 
These descriptions of curvature bounds are also a topic of substantial interest in Alexandrov geometry. 
Indeed, for metric spaces there is an abundance of different formulations for (sectional) curvature bounds, cf.\ \cite{BBI01, BH99, AKP19}. 

Some of these have been added to the Lorentzian repertoire as well, such as, for example, the so-called monotonicity condition. 
In fact, both \cite{BS22} and \cite{BMS22}, more or less simultaneously, introduced the concept of hyperbolic angles into the synthetic Lorentzian theory, and gave a formulation of timelike curvature bounds expressed via the monotonic behaviour of angles. 
In \cite{BMS22}, there is also a formulation using angles directly in relation to their comparison angles, but only for lower curvature bounds, and in \cite{BS22}, angle comparison is only obtained as an implication of ordinary curvature bounds and not vice versa.

As is evident from Alexandrov geometry, having a wide array of different characterizations at one's disposal is vital to producing a rich and flourishing theory of synthetic geometry. 
In this work we collect all the currently known approaches to (sectional) curvature bounds in Lorentzian pre-length spaces, add several new ones, and prove equivalence of all these notions under suitable assumptions. 
The precise interdependence of all of these concepts will be established in Theorem \ref{thm:allEquiv}. 

\section{Preliminaries}
Since the theory of Lorentzian pre-length spaces is by now quite well established, we are not going to repeat the basic definitions, instead referring to \cite{KS18} for the fundamentals, and to \cite{BS22, BMS22} for more in-depth discussions of angles and some of the curvature bounds discussed here. 
\begin{rem}[Notations and conventions]
For the sake of readability and consistency in notation, we collect some generalities here. Let $X$ denote a \LpLSn. 
\begin{enumerate} 
    \item Unless explicitly stated otherwise, $K$ is assumed to be any real number. 
    It symbolizes the curvature of the model space, and hence it is only explicitly mentioned when necessary. 
    
    \item By a \emph{distance realizer} we mean a curve that attains the $\tau$-distance between its endpoints, i.e., if $\gamma$ is a causal curve from $x$ to $y$, then $\gamma$ is a distance realizer if $\tau(x,y)=L_{\tau}(\gamma)$. 
    Instead of labeling the curve, we may also use the notation $[x,y]$ for a distance realizer from $x$ to $y$. 
    
    \item By a \emph{timelike triangle} in $X$, we mean a collection of three timelike related points $x \ll y \ll z, \tau(x,z) < \infty$, and three distance realizers pairwise joining them. By an \emph{admissible causal triangle} we mean a collection of three points  $x \leq y \ll z$ or $x \ll y \leq z, \tau(x,z) < \infty$, together with distance realizers between timelike related points. In other words, one of the two short sides is allowed to be null, and if it is, it need not be realized by a curve. When the context allows for it, we may refer to either of these as just a triangle, and both are denoted by $\Delta(x,y,z)$. 

    \item Given a triangle $\Delta(x,y,z)$ in $X$, a \emph{comparison triangle} is a triangle in $\lm{K}$ with the same side lengths. The existence of comparison triangles is established in the so-called Realizability Lemma, see \cite[Lemma 2.1]{AB08}. 

    \item We say that a triangle or hinge (cf.\ Definition \ref{def: hinge} below) in $X$ \emph{satisfies size bounds for $K$} if there exists a (unique) comparison configuration in $\lm{K}$, the Lorentzian model space of constant curvature $K$. Throughout the paper, we assume any such configuration to satisfy size bounds. As will become evident later on, this is precisely the case when the largest $\tau$-value in this configuration is less than $D_K$. 

    \item To increase readability, we are going to mark points arising in comparison triangles with a bar, in comparison hinges with a tilde, and in four-point comparison configurations (cf.\ Definition \ref{def:four-point configs} below) with a hat. 
    That is, if $x \in X$, then the corresponding point in a comparison triangle will be denoted by $\bx$, the corresponding point in a comparison hinge will be denoted by $\tx$, and the corresponding point in a four-point comparison configuration will be denoted by $\hx$. 
    To ease the notational burden we will, however, not mark the time separation function in the comparison spaces by $\bar{\tau}$. 
    Instead, we will always just write $\tau$ since the marking of the arguments will clearly indicate when the time separation function in the model space is considered. 
    
    \item We will use the slightly updated version of curvature bounds introduced in \cite{BNR23}. 
    In order for these conditions to be non-void, i.e., trivially satisfied by pathological spaces where, e.g., $\tau \equiv \infty$, we will always assume $X$ to be chronological. 
    This is not a substantial restriction since any space that satisfies some timelike curvature bound in the original formulation in \cite[Definition 4.7]{KS18} is chronological by definition, as there $\tau$ was supposed to be finite on comparison neighbourhoods. 
    
    \item The notation regarding hyperbolic angles (which are defined as limits) will use curves, while angles in the comparison spaces will use the endpoints of the corresponding curves (as the model spaces are uniquely geodesic, this is consistent). 

    \item Throughout this paper, we introduce several formulations of curvature bounds, all of which are built on the notion of comparison neighbourhoods. We will adhere to the following terminology: a \LpLS is said to have \emph{curvature bounded below (resp.\ above) by $K$} in any of these senses if it is covered by the corresponding $(\geq K)$- (resp.\ $(\leq K)$-)comparison neighbourhoods. Moreover, we say $X$ has \emph{curvature globally bounded below (resp.\ above) by $K$} in any of these senses if $X$ is a corresponding $(\geq K)$- (resp.\ $(\leq K)$-)comparison neighbourhood. 
    In general, implications on comparison neighbourhoods without further assumptions on the neighbourhoods yield implications on curvature bounds for $X$, so that we will not highlight this in every statement. 
    However, if this only holds under additional assumptions, we will explicitly formulate the implication for curvature bounds on $X$.
\end{enumerate}

\end{rem}
For completeness, we briefly repeat the most important definitions surrounding hyperbolic angles, following \cite{BS22}. 
\begin{defi}[$K$-comparison angles and sign]
Let $X$ be a \LpLSn, $\Delta(x,y,z)$ an admissible causal triangle in $X$, and $\Delta(\bx,\by,\bz)$ a comparison triangle in $\lm{K}$ for $\Delta(x,y,z)$. Assume that, say, $x$ is adjacent to two timelike sides (the following definition clearly works for any other vertex who is between two timelike sides). 
\begin{itemize}
    \item[(i)] The \emph{$K$-comparison angle} at $x$ is defined as the ordinary hyperbolic angle at $\bx$ between $\by$ and $\bz$:
\begin{equation}
    \tilde{\ma}_x^{K}(y,z):=\ma_{\bx}^{\lm{K}}(\by,\bz)=\arcosh(|\langle \gamma_{\bx \by}'(0), \gamma_{\bx \bz}'(0) \rangle|) \, , 
\end{equation}
where we assume the mentioned geodesics to be unit speed parametrized. 
\item[(ii)] The \emph{sign} $\sigma$ of a $K$-comparison angle is the sign of the corresponding inner product (in the $-,+,\cdots,+$ convention). That is, in this notation, the sign is $-1$ if the angle is measured at $x$ or $z$ and $1$ if the angle is measured at $y$.
\item[(iii)] The \emph{signed $K$-comparison angle} is defined as $\tilde{\ma}_x^{K,S}(y,z):=\sigma \tilde{\ma}_x^{K}(y,z)$.
\end{itemize}
\end{defi}

\begin{defi}[Angles]
    Let $X$ be a \LpLS and let $\alpha$ and $\beta$ be two timelike curves of arbitrary time orientation emanating %at 
    from 
    $\alpha(0)=\beta(0) \eqqcolon x$. 
    \begin{enumerate}
    \item The \emph{angle} between $\alpha$ and $\beta$, if it exists, is defined as 
    \begin{equation}
    \label{eq: angle}
        \ma_x(\alpha,\beta) \coloneqq \lim_{s,t \to 0}\tilde{\ma}_x^{0}(\alpha(s),\beta(t)) \, , 
    \end{equation}
    where the limit only takes values of $s$ and $t$ into account for which the triple $(x, \alpha(s),\beta(t)) $ (or some permutation thereof) forms an admissible causal triangle.\footnote{As admissible causal triangles arise as limits of timelike triangles, one can also restrict this to timelike triangles.} 
    \item The \emph{sign} $\sigma$ of an angle is $-1$ if $\alpha$ and $\beta$ have the same time orientation and $1$ otherwise. The \emph{signed angle} is defined as $\ma_x^S(\alpha,\beta) \coloneqq \sigma \ma_x(\alpha,\beta)$. 
    \end{enumerate}
\end{defi}
\begin{rem}
Note that one could also look at angles defined using any $K$ instead of $0$ in \eqref{eq: angle}. 
However, the limit is the same regardless of the model space in which it is considered due to \cite[Proposition 2.14]{BS22}. Although this reference uses strong causality to ensure size bounds, this is not actually necessary, the condition from Definition \ref{def-cb-tr}(i) is sufficient: As $\tau$ is continuous near $(p,p)$ (since $\tau(p,p)=0 < D_K$) and $\tau^{-1}([0,D_K))$ is open, for two curves $\alpha,\beta$ emanating from $p$, $(\alpha(t),\beta(t))$ will initially stay in $\tau^{-1}([0,D_K))$. This is why we will usually drop the superscript $K$ in the comparison angle and just write $\tilde{\ma}_x(y,z)$. 
Similarly, we will write $\ma_{\bx}(\by,\bz)$ instead of $\ma_{\bx}^{\lm{K}}(\by,\bz)$ for angles between points in $\lm{K}$ (which do not necessarily arise as comparison angles). 
\end{rem}
\begin{defi}[Hinges and comparison hinges]
\label{def: hinge}
    Let $X$ be a \LpLSn. Let $\alpha: [0,a] \to X$ and $\beta: [0,b] \to X$ be two timelike distance-realizers emanating from the same point $x:=\alpha(0)=\beta(0)$. 
    Then we say that $\alpha$ and $\beta$ and their associated angle form a \emph{hinge\footnote{In \cite[Definition 2.11]{BS22}, such a constellation is only called a hinge if the angle exists, meaning that the $\limsup$ is a limit and is finite. Here, we drop these restrictions, see Definition \ref{def-cb-hinge}(iv) below.} at $x$}, and denote it by
    %By the \emph{hinge formed by $\alpha$ and $\beta$}, denoted 
    $(\alpha, \beta)$.
    In particular, if the angle $\ma_x(\alpha, \beta)$ is finite, then by a \emph{comparison hinge} in $\lm{K}$, denoted\footnote{As the sides of a comparison hinge are unique, we may also denote a comparison hinge by its endpoints (and the vertex where the angle is measured), i.e., $(\tilde{\alpha}(a), \tx, \tilde{\beta}(b))$.} by $(\tilde \alpha, \tilde \beta)$, we mean a constellation $\tx$ together with two distance-realizers $\tilde{\alpha}$ and $\tilde{\beta}$ emanating from that point, such that they have the same length (and time-orientation) as $\alpha$ and $\beta$, respectively, and such that the angle between them is the same, i.e., $\ma_x^S(\alpha, \beta)=\ma_{\tx}^S(\tilde{\alpha}(a), \tilde{\beta}(b))$.
\end{defi}
Another concept we shall require is the so-called finite diameter of a \LpLSn. Introduced in \cite{BNR23}, this number essentially bounds the length of geodesics for which it is possible to implement comparison methods. 
\begin{defi}[Finite diameter]
    Let $X$ be a \LpLSn. 
    \begin{enumerate}
        \item     The \emph{finite diameter} of $X$ is 
    \begin{equation}
        \diamfin=\sup (\{ \tau(x,y) \mid x \ll y \} \setminus \{ \infty \}) \, , 
    \end{equation}
    i.e., the supremum of all values $\tau$ takes except $\infty$. 
    \item By $D_K$ we denote the finite diameter of $\lm{K}$. In particular, 
    \begin{equation}
        D_K=\diamfin(\lm{K})= 
        \begin{cases}
            \infty, & \text{ if } K \geq 0 \, , \\
            \frac{\pi}{\sqrt{-K}}, & \text{ if } K < 0 \, .
        \end{cases}
    \end{equation}
    \end{enumerate}
\end{defi}
Note that the formula of $D_K$ is pleasantly similar to the diameter of the Riemannian model spaces in metric geometry. 
Finally, we introduce the concept of being (locally) $r$-geodesic, another notion very similar to its metric counterpart. 
\begin{defi}[$r$-geodesic]
\label{def: r-geodesic}
Let $X$ be a \LpLS and let $0<r\leq \infty$. 
\begin{enumerate}
    \item A subset $U \subseteq X$ is called $r$-geodesic if for all $p \ll q$ in $U$ with $\tau(p,q)<r$ there exists a distance realizer in $U$ connecting them. In particular, if $U$ is $\infty$-geodesic, then this corresponds to the original definition of being geodesic, cf.\ \cite[Definition 3.27]{KS18}, without the assumption of the existence of null realizers. 
    \item If one additionally assumes the existence of null realizers, we call $U$ \emph{causally $r$-geodesic}.
    \item $X$ is called \emph{locally (causally) $r$-geodesic} if every point has a neighbourhood which is (causally) $r$-geodesic. 
\end{enumerate}
\end{defi}
\section{Curvature comparison for \LpLSn s}
\label{sec-cur-com}
In this section, we recall and amend the currently known characterizations of timelike curvature bounds. In particular, for those cases where, as of yet, only implications between certain notions are available, we show full equivalence. Timelike curvature bounds in the setting of Lorentzian pre-length spaces were first introduced in \cite[Definition 4.7]{KS18} and slightly updated in \cite[Definition 2.7]{BNR23}. 
\begin{defi}[Curvature bounds by timelike triangle comparison]
\label{def-cb-tr}
Let $X$ be a \LpLSn. An open subset $U$ is called a $(\geq K)$- (resp.\ $(\leq K)$-)comparison neighbourhood) in the sense of \emph{timelike triangle comparison} if:
\begin{enumerate}
\item $\tau$ is continuous on $(U\times U) \cap \tau^{-1}([0,D_K))$, and this set is open. 
\item $U$ is $D_K$-geodesic.
\item \label{TLCB.item3} Let $\Delta (x,y,z)$ be a timelike triangle in $U$, with $p,q$ two points on the sides of $\Delta (x,y,z)$. Let $\Delta(\bar{x}, \bar{y}, \bar{z})$ be a comparison triangle in $\lm{K}$ for $\Delta (x,y,z)$ and $\bar{p},\bar{q}$ comparison points for $p$ and $q$, respectively. Then 
\begin{equation}
\label{eq: timelike triangle comparison inequality}
\tau(p,q) \leq \tau(\bar{p},\bar{q}) \quad \text{ (resp.\ } \tau(p,q) \geq \tau(\bar{p}, \bar{q}) \text{)} \, .
\end{equation}
\end{enumerate}
Note that within a $(\geq K)$-comparison neighbourhood, $p \ll q$ implies $\bar{p} \ll \bar{q}$, and within a $(\leq K)$-comparison neighbourhood, $\bar{p} \ll \bar{q}$ implies $p \ll q$. 
\end{defi}

As a first different formulation, we mention one-sided triangle comparison. This was originally introduced in \cite{BMS22} for lower curvature bounds only, but can easily be adapted to work for upper curvature bounds as well. 

\begin{defi}[Curvature bounds by one-sided timelike triangle comparison]
\label{def-cb-tr-onesided}
Let $X$ be a \LpLSn. An open subset $U$ is called a $(\geq K)$- (resp.\ $(\leq K)$-)comparison neighbourhood) in the sense of \emph{one-sided timelike triangle comparison} if:
\begin{enumerate}
\item $\tau$ is continuous on $(U\times U) \cap \tau^{-1}([0,D_K))$, and this set is open. 
\item $U$ is $D_K$-geodesic.
\item \label{1TLCB.item3} Let $\Delta (x,y,z)$ be a timelike triangle in $U$. Let $p$ be a point on one side of the triangle and denote by $v \in \{x,y,z\}$ the vertex opposite of $p$. 
Let $\Delta(\bar{x}, \bar{y}, \bar{z})$ be a comparison triangle in $\lm{K}$ for $\Delta (x,y,z)$ and let $\bar{p}$ be a comparison point for $p$. Then\footnote{Note that as $X$ is supposed to be chronological, at most one of the $\tau$-values in $X$ is positive in each case. } 
\begin{align*}
& \tau(p,v) \leq \tau(\bar{p},\bar{v}) \text{ and } \tau(v,p) \leq  \tau(\bar{v}, \bar{p}) \\
\quad \text{ (resp.\ } 
& \tau(p,v) \geq \tau(\bar{p},\bar{v}) \text{ and } \tau(v,p) \geq \tau(\bar{v}, \bar{p})
\text{)} \, .
\end{align*}
\end{enumerate}
\end{defi}

\begin{pop}[One-sided triangle comparison]
\label{rem: one-sided triangle comparison and triangle comparison are equivalent}
Let $U$ be an open subset in a \LpLS $X$. 
Then $U$ is a $(\geq K)$- (resp.\ $(\leq K)$-)comparison neighbourhood in the sense of timelike triangle comparison if and only if it is a $(\geq K)$- (resp.\ $(\leq K)$-)comparison neighbourhood in the sense of one-sided timelike triangle comparison. 
\end{pop}
\begin{proof}
Concerning the non-trivial direction, the case for $(\leq K)$-comparison neighbourhoods is shown in \cite[Proposition 4]{BMS22}, so we only need to consider upper curvature bounds. 
Moreover, if, say, $p \in [x,y]$ and $q \in [y,z]$, then the proof of \cite[Proposition 4]{BMS22} works for $(\geq K)$-comparison neighbourhoods in complete analogy, by reversing the inequalities between $\tau$-values, as well as between the nonnormalized angles established in \cite[Lemmas 1 \& 2]{BMS22}. 

The other case of one point being on the longest side is where the cases of lower and upper curvature bounds differ (and this is in fact also the reason why we, in contrast to \cite{BMS22}, require the $\tau$-inequalities to hold even if there is no timelike relation between a point and the opposing vertex). 
Let, say, $p \in [x,y]$ and $q \in [x,z]$. 
If $\tau(\bp,\bq)=0$, then \eqref{eq: timelike triangle comparison inequality} is trivially satisfied. So assume $\bp \ll \bq$. Let $\Delta(\bx,\by,\bz)$ be a comparison triangle for $\Delta(x,y,z)$ and let $\Delta(\bx',\bp',\bz')$ be a comparison triangle for the subtriangle $\Delta(x,p,z)$. 
Since $U$ is a $(\geq K)$-comparison neighbourhood in the sense of one-sided timelike triangle comparison, we infer $\tau(p,z) \geq \tau(\bp,\bz)$, hence also $\tau(\bp',\bz')=\tau(p,z) \geq \tau(\bp,\bz)$. 
The triangles $\Delta(\bx',\bp',\bz')$ and $\Delta(\bx, \bp,\bz)$ have two sides of equal length and an inequality between the lengths of their third sides, so we obtain $\ma_{\bx}(\bp,\bz) \geq \ma_{\bx'}(\bp',\bz')$ by law of cosines monotonicity, cf.\ \cite[Remark 2.5]{BS22}. 
Clearly, this further yields 
\begin{equation}
\ma_{\bx}(\bp,\bq) = \ma_{\bx}(\bp,\bz) 
\geq \ma_{\bx'}(\bp',\bz') = \ma_{\bx'}(\bp',\bq') \, .
\end{equation} 
Applying law of cosines monotonicity once more to the subtriangles $\Delta(\bx,\bp,\bq)$ and $\Delta(\bx',\bp',\bq')$ of $\Delta(\bx,\bp,\bz)$ and $\Delta(\bx',\bp',\bz')$, respectively, we obtain $\tau(\bp,\bq) \leq \tau(\bp',\bq')$. By one-sided comparison in the subtriangle $\Delta(x,p,z)$ of $\Delta(x,y,z)$, we get $\tau(p,q) \geq \tau(\bp',\bq')$, hence the desired inequality of $\tau(p,q) \geq \tau(\bp,\bq)$ follows. 
Assuming the opposite timelike relation of $\bq \ll \bp$, the proof works out just the same.
\end{proof}

Throughout this work, the property of a \LpLS being \emph{locally causally closed}, cf.\ \cite[Definition 3.4]{KS18}, is used several times. This should be compared with the slightly weaker notion of being \emph{locally weakly causally closed}, which was introduced in \cite[Definition 2.19]{ACS20} to better resemble the smooth setting below the level of strong causality. It turns out that these two notions are equivalent under the assumption of strong causality, see \cite[Proposition 2.21]{ACS20}. 
In a similar spirit, the following lemma shows in what way one could also work with this latter definition in the present paper. 

\begin{lem}[Causally closed comparison neighbourhoods]
\label{lem:causally_closed_comp_nbhds}
Let $X$ be a locally weakly causally closed and strongly causal \LpLS which has curvature bounded below (resp.\ above) by $K$ in the sense of timelike triangle comparison. 
Then each point has a comparison neighbourhood which is causally closed. 
This will also work in any of the following senses, except the four-point condition. 
\end{lem}
\begin{proof}
Let $x\in X$, and let $U_1$ be a weakly causally closed neighbourhood. 
Let $U_2$ be a curvature comparison neighbourhood. 
 According to (i) in Definition \ref{def-cb-tr}, we have that $(U_2\times U_2)\cap \tau^{-1}([0,D_K))$ is open, so we find $U_3$ with $x\in U_3\subseteq U_2$ and $U_3\times U_3\subseteq \tau^{-1}([0,D_K))$. Then by strong causality, there exists a causally convex set $V=\cap_{i=1}^n I(p_i,q_i)$ with $x\in V\subseteq U_1\cap U_3$. We claim that $V$ is a causally closed comparison neighbourhood. 
The properties (i) and (iii) follow as we are just restricting. 
For (ii), note that distance realizers in $U_2$ with endpoints in $V$ are automatically contained in $V$ by causal convexity. 
To see causal closure, let $x_n \leq y_n$ in $V$ with $x_n \to x$ and $y_n \to y$. 
As $x_n, y_n \in U_3$, we have $\tau(x_n, y_n)<D_K$, so as $U_2$ is $D_K$-geodesic, there is a causal distance realizer $\gamma_n$ connecting $x_n$ and $y_n$. 
In particular, there is a causal curve joining them contained in $V$, which in the the terminology of \cite{ACS20} is denoted by $x_n \leq_{V} y_n$, from which it follows by the weak causal closure that $x \leq_{V} y$, i.e., $x$ and $y$ are also joined by a causal curve in $V$. 
Thus, we obtain $x \leq y$ and hence $V$ is causally closed. 
\end{proof}

The following three characterizations were all to some extent developed in \cite{BS22} and use the notion of being \emph{locally strictly timelike geodesically connected}, see \cite[Definition 1.12]{BS22}. 
It turns out that this is equivalent to the seemingly more natural property of being regular, which is why we will use this formulation instead (see Lemma \ref{lem: regular and loc str tl geod conn} below). 
The following concept of regularity for Lorentzian pre-length spaces was introduced in \cite[Definition 2.4]{BNR23} (and has to be compared to the notion
of regularity of Lorentzian length spaces, cf.\ \cite[Definition 3.22]{KS18}). The two concepts are equivalent under the asssumption of strong causality, see Lemma \ref{lem: regularly localizable vs regular and localizable} below. 
\begin{defi}[Regularity]
\label{def: regularity}
    A \LpLS $X$ is called \emph{regular} if every distance realizer between timelike related points is timelike, i.e., it cannot contain a null segment. 
\end{defi}
\begin{lem}[Regularly localizable vs.\ regular and localizable]
\label{lem: regularly localizable vs regular and localizable}
A strongly causal \LpLS is regularly localizable if and only if it is regular and localizable.
\end{lem}
\begin{proof}
A regularly localizable \LpLS is automatically regular in the sense of Definition \ref{def: regularity}. By \cite[Lemma 4.3]{GKS19}, we can choose localizable neighbourhoods $\Omega$ such that the local time separation function agrees with $\tau$. Thus, also a distance realizer w.r.t.\ the local time separation function is a global distance realizer, and has to stay timelike. 
\end{proof}

\begin{lem}[Regularity and local strictly timelike geodesic connectedness]
\label{lem: regular and loc str tl geod conn}
A $D_K$-geodesic \LpLS $X$ such that $t \mapsto L(\gamma|_{[0,t]})$ is continuous\footnote{This can for example be achieved by assuming that any point has a neighbourhood $U$ such that $\tau|_{U \times U}$ is continuous, which is the case in any space with curvature bounds, see \cite[Lemma 3.33]{KS18} (whose proof also works with $\tau$ being merely continuous in this sense).} for every distance realizer $\gamma:[0,b] \to X$ is regular if and only if it is locally strictly timelike geodesically connected.
\end{lem}
\begin{proof}
If $X$ is regular, any distance realizer between timelike related points is timelike by definition. 

If $X$ is locally strictly timelike geodesically connected, let $x\ll y$ and $\gamma:[0,b]\to X$ a causal distance realizer connecting $x$ to $y$. The function $f(t) \eqqcolon L(\gamma|_{[0,t]}) = \tau(x,\gamma(t))$ is continuous and monotonically increasing and, as $\gamma$ is a distance realizer, we have $\tau(\gamma(s),\gamma(t))=f(t)-f(s)$ for $s\leq t$. If $\gamma$ is not timelike, we have $s<t$ such that $\tau(\gamma(s),\gamma(t))=0$, so $f$ is constant on $[s,t]$, but $f(0)<f(b)$. W.l.o.g.\  assume $f(0)<f(s)$ and let $s$ be minimal such that $f$ is constant on $[s,t]$. Then we have $f(s-\varepsilon)<f(s)=f(s+\varepsilon)$. 
% regarding $\gamma$ this means that at 
Thus, at $p=\gamma(s)$ we have a distance realizer $\gamma|_{[s-\varepsilon,s+\varepsilon]}$ with $\gamma(s-\varepsilon)\ll\gamma(s+\varepsilon)$, but $\gamma(s)\not\ll\gamma(s+\varepsilon)$, contradicting the assumption that $X$ is strictly timelike geodesically connected in a neighbourhood of $p$.
\end{proof}
Next we turn to the so-called monotonicity condition. This equivalent formulation was introduced in \cite[Definition 4.9]{BS22} and updated in \cite[Definition 2.15]{BNR23}. Intuitively, it says that signed comparison angles cannot increase (decrease) when approaching the vertex. Note that at first glance, this seems opposite to (iii) in the following definition, but this apparent contradiction is resolved by noting that increasing the inputs in $\theta$ causes one to move away from the vertex. 
\begin{defi}[Curvature bounds by monotonicity comparison]
\label{def-cb-mon}
Let $X$ be a regular \LpLSn. An open subset $U$ is called a $(\geq K)$- (resp.\ $(\leq K)$-)comparison neighbourhood in the sense of \emph{monotonicity comparison} if:
\begin{enumerate}
\item $\tau$ is continuous on $(U\times U) \cap \tau^{-1}([0,D_K))$, and this set is open.
\item $U$ is $D_K$-geodesic. 
\item 
\label{def-cb-mon.main} 
Let $\alpha:[0,a]\to U,\beta:[0,b]\to U$ be timelike distance realizers such that $x:=\alpha(0)=\beta(0)$ and such that $L(\alpha), L(\beta)$, $\tau(\alpha(a), \beta(b))$, $\tau(\beta(b), \alpha(a))$ $< D_K$. 
Let $\theta:D\to[0,\infty)$ be defined by $\theta(s,t):=$ $\tilde{\ma}_x^{K,\mathrm{S}}(\alpha(s),\beta(t))$ ($D\subseteq (0,a]\times(0,b]$ is the set where this is defined, i.e., the set of points where there is some causal relation between $\alpha(s)$ and $\beta(t)$ and the comparison triangle exists). Then $\theta$ is monotonically increasing (resp.\ decreasing).
\end{enumerate}
\end{defi}

The equivalence between monotonicity comparison and triangle comparison has already been established in \cite{BS22}:

\begin{pop}[Equivalence of triangle and monotonicity comparison]
\label{pop-eqiv-mon-tr}
Let $U\subseteq X$ be an open subset in a regular \LpLS $X$. 
Then $U$ is a $(\geq K)$- (resp.\ $(\leq K)$-)comparison neighbourhood in the sense of timelike triangle comparison if and only if it is a $(\geq K)$- (resp.\ $(\leq K)$-)comparison neighbourhood in the sense of monotonicity comparison. 
\end{pop}
\begin{proof}
See \cite[Theorem 4.13]{BS22}.
\end{proof}

\begin{rem}[One-sided monotonicity comparison]
\label{rem-cb-one-sided-mon}
    As it turns out, similar to the case of triangle comparison, there is also a one-sided version of monotonicity comparison. 
    This means that we leave the parameter of one of the curves fixed. 
    Clearly, ordinary monotonicity comparison implies one-sided monotonicity comparison. 
    Conversely, one simply varies the parameters one after the other (taking care to make the points stay timelike related), to get from one-sided monotonicity comparison to the original formulation. 
    As will be seen below, the one-sided version is more convenient to work with. 
\end{rem}

The next formulation of curvature bounds is closely related to monotonicity comparison. Instead of talking about monotonic behaviour of comparison angles when going along the sides of a hinge, we now require an inequality between hyperbolic angles and comparison angles. A formulation of curvature bounds using angles was first introduced in \cite{BMS22}. However, this uses a slightly different definition for angles. The following definition is better suited for our setting. Note that as in \cite{BMS22}, for the case of upper curvature bounds one needs to explicitly assume one case of the triangle inequality for angles.

\begin{defi}[Curvature bounds by angle comparison]
\label{def-cb-ang}
Let $X$ be a regular \LpLSn. An open subset $U$ is called a $(\geq K)$- (resp.\ $(\leq K)$-)comparison neighbourhood in the sense of \emph{angle comparison} if:
\begin{enumerate}
\item $\tau$ is continuous on $(U\times U) \cap \tau^{-1}([0,D_K))$, and this set is open.
\item $U$ is $D_K$-geodesic.
\item \label{def-cb-ang.main} Let $\alpha:[0,a]\to U,\ \beta:[0,b]\to U$ be distance realizers such that $L(\alpha), L(\beta), \tau(\alpha(a), \beta(b)), \tau(\beta(b), \alpha(a)) < D_K$ and such that $x \coloneqq \alpha(0)=\beta(0)$ and $\alpha(a)$, $\beta(b)$ are causally related. 
Then 
\begin{equation}
\ma_x^{\mathrm{S}}(\alpha,\beta)\leq\tilde{\ma}_x^{K,\mathrm{S}}(\alpha(a),\beta(b)) \quad \text{ (resp.\ } \ma_x^{\mathrm{S}}(\alpha,\beta)\geq\tilde{\ma}_x^{K,\mathrm{S}}(\alpha(a),\beta(b)) \text{)} \, .
\end{equation}
\item For $(\geq K)$-comparison neighbourhoods only: let $\alpha,\beta,\gamma:[0,\varepsilon)\to U$ be distance realizers, all emanating from the same point $x:=\alpha(0)=\beta(0)=\gamma(0)$. 
Suppose $\alpha$ and $\gamma$ have the same time-orientation and $\beta$ has the opposite time-orientation. Then we have the following special case of the triangle inequality of angles:
\begin{equation}
\label{eq: triangle inequality of angles lower curvature bounds}
\ma_x(\alpha,\gamma)\leq\ma_x(\alpha,\beta)+\ma_x(\beta,\gamma) \, .
\end{equation}
\end{enumerate}
\end{defi}
Note that for a \LpLS $X$ with curvature bounded below, in point $(iv)$ one can also take the curves as maps into $X$, since angles only depend on the initial segments of the curves anyways. 
\begin{rem}[Only considering timelike triangles] 
\label{rem:timelikeTriangles}
While Definition \ref{def-cb-ang} requires all admissible causal triangles to satisfy the angle condition at each vertex where an angle is defined, it is often more convenient to work only with timelike triangles. However, it becomes clear, when using our new vocabulary, that only requiring the angle condition to hold at each vertex of every timelike triangle is an equivalent constraint.

Consider, for example, an admissible causal triangle $\Delta(x,y,z)$ where $x \ll y \leq z$, with $\tau(y,z)=0$ with a failing angle condition. (The case $x \leq y \ll z$ is similar.) We show that moving $y$ slightly can create a timelike triangle with a failing angle condition.
Let $\alpha: [0,a] \to X$ be a distance realizer from $x$ to $y$ and $\beta: [0,b] \to X$ a distance realizer from $x$ to $z$. By regularity, the side $[y,z]$ contains no timelike segments, so the only angle which is defined in $\Delta(x,y,z)$ is at $x$. It follows that, if an angle condition fails, it necessarily does so at $x$. 
As $s \nearrow a$ the triangle $\Delta(x, \alpha(s), z)$ is timelike and converges to the original $\Delta(x,y,z)$. 
The signed angle $\ma_x^{\mathrm{S}}(\alpha,\beta)$ is  not dependent on the endpoint of $\alpha$, while the signed comparison angles $\tilde{\ma}_x^{K,\mathrm{S}}(\alpha(s),\beta(b))$ vary continuously with $s$.
The failure of the angle condition at $x$ is an open condition and so, for $s$ sufficiently close to $a$, the timelike triangle $\Delta(x, \alpha(s), z)$ also has a failing angle condition at $x$. 

Hence, the existence of an admissible causal triangle with failing angle condition implies the existence of a timelike triangle (of comparable size) with failing angle condition. The contrapositive then tells us that a space has a curvature bound with respect to angle comparison in admissible causal triangles if it does so with respect to angle comparison in timelike triangles. The converse implication is tautological and the two notions are therefore equivalent. In particular, we refrain from introducing ``causal angle comparison'', since it would anyways be automatically equivalent to Definition \ref{def-cb-ang}. 
\end{rem}
We now show that monotonicity comparison implies angle comparison. As \cite{BMS22} uses a different convention and the proof is elementary, we give it anew. We intend to form an implication circle, so the converse implication is proven later. In that proof, a technical detail requires us to assume the triangle inequality of angles as displayed in %Definition \ref{def-cb-ang}(iv)
\eqref{eq: triangle inequality of angles lower curvature bounds}, which was achieved in \cite[Theorem 4.5(ii)]{BS22} using \emph{geodesic prolongation}, cf.\ \cite[Definition 4.2]{BS22}. 

As this is a rather strong property, however, we believe that it is in fact more natural to directly impose this condition, whenever necessary. We will do this by saying that $X$ satisfies \eqref{eq: triangle inequality of angles lower curvature bounds}.

\begin{pop}[Monotonicity comparison implies angle comparison]
\label{pop-implication-mon-->angle} 
Let $U\subseteq X$ be an  open subset in a regular \LpLSn. 
In the case of lower curvature bounds, additionally assume that $X$ satisfies \eqref{eq: triangle inequality of angles lower curvature bounds}. 
Then if $U$ is a $(\geq K)$- (resp.\ $(\leq K)$-)comparison neighbourhood in the sense of monotonicity comparison, it is also a $(\geq K)$- (resp.\ $(\leq K)$-)comparison neighbourhood in the sense of angle comparison. 
\end{pop}
\begin{proof}
(i) and (ii) are the same in both definitions, and (iv) in the case of lower curvature bounds is assumed directly.  
So the only point to check is (iii) in Definition \ref{def-cb-ang}. 
To this end, given distance realizers $\alpha$ and $\beta$ in $U$ as in Definition \ref{def-cb-ang}(iii), we have by definition
\begin{equation}
    \ma_x^{\mathrm{S}}(\alpha,\beta) = \lim_{s,t \to 0} \tilde{\ma}_x^{K,\mathrm{S}}(\alpha(s),\beta(t)) = \lim_{s,t \to 0}\theta(s,t) \, .
\end{equation}
As $\theta$ is monotonous by assumption, the desired inequality holds in the limit also. 
\end{proof}

Finally, we turn to  
hinge comparison. This uses the construction of hinges and comparison hinges, and the distinguishing
inequality pertains to the opposite side of the angle 
that forms the hinge. As will be seen below, this is closely related to angle comparison via the law of cosines. To establish a proper equivalence between hinge comparison and angle comparison, however, we need to additionally assume the same case of the triangle inequality of angles as in Definition \ref{def-cb-ang}. 
\begin{defi}[Curvature bounds by hinge comparison]
\label{def-cb-hinge}
Let $X$ be a regular \LpLSn. An open subset $U$ is called a $(\geq K)$- (resp.\ $(\leq K)$-)comparison neighbourhood in the sense of \emph{hinge comparison} if:
\begin{enumerate}
\item $\tau$ is continuous on $(U\times U) \cap \tau^{-1}([0,D_K))$, and this set is open.
\item $U$ is $D_K$-geodesic.
\item Let $\alpha:[0,a]\to U,\ \beta:[0,b]\to U$ be distance realizers emanating from the same point $x=\alpha(0)=\beta(0)$ such that $L(\alpha), L(\beta), \tau(\alpha(a), \beta(b))$, $\tau(\beta(b), \alpha(a)) < D_K$ and such that the angle $\ma_x(\alpha,\beta)$ is finite. Let $(\tilde{\alpha}, \tilde{\beta})$ form a comparison hinge for $(\alpha, \beta)$ in $\lm{K}$. Then 
\begin{equation}
\tau(\alpha(a),\beta(b)) \geq \tau(\tilde{\alpha}(a),\tilde{\beta}(b)) \ \, \text{ (resp.\ } \tau(\alpha(a),\beta(b)) \leq \tau(\tilde{\alpha}(a),\tilde{\beta}(b)) \text{)} \, .
\end{equation}
\item Let $\alpha,\beta$ be as in (iii), without the restriction of finite angle. For $(\geq K)$-comparison neighbourhoods we assume that if $\alpha,\beta$ point in different time directions the angle can never be infinite, and for $(\leq K)$-comparison neighbourhoods we assume that if $\alpha,\beta$ point in the same time directions the angle can never be infinite. \footnote{This can be viewed as the limit of (iii) as $\ma_x(\alpha,\beta) \to \infty$ and agrees with \cite[Lemma 4.10]{BS22}. The rationale behind (iv) is to avoid the case of curvature bounds being trivially satisfied when angles are infinite.}
\item for $(\geq K)$-comparison neighbourhoods only: let $\alpha,\beta,\gamma:[0,\varepsilon)\to U$ be distance realizers all emanating from the same point $x:=\alpha(0)=\beta(0)=\gamma(0)$. 
Suppose $\alpha$ and $\gamma$ have the same time-orientation and $\beta$ has the opposite time-orientation. Then we have the following special case of the triangle inequality of angles:
\begin{equation}
\ma_x(\alpha,\gamma)\leq\ma_x(\alpha,\beta)+\ma_x(\beta,\gamma) \, .
\end{equation}
\end{enumerate}
\end{defi}

\begin{pop}[Equivalence of angle and hinge comparison]
\label{prop:equivalence_hinge_angle}
Let $U$ be an open subset in a regular \LpLS $X$. 
Then $U$ is a $(\geq K)$- (resp.\ $(\leq K)$-)comparison neighbourhood in the sense of angle comparison if and only if it is a $(\geq K)$- (resp.\ $(\leq K)$-)comparison neighbourhood in the sense of hinge comparison. 
\end{pop}
\begin{proof}
    Definition \ref{def-cb-ang}(iv) and Definition \ref{def-cb-hinge}(v) as well as (i) and (ii) in both formulations are the same. 
    Thus, only the case of (iii) in both conditions as well as \ref{def-cb-hinge}(iv) are of interest. 
    Concerning \ref{def-cb-hinge}(iv), note that angle comparison for, say, lower curvature bounds, yields 
    \begin{equation}
        \ma_x^{\mathrm{S}}(\alpha,\beta)\leq\tilde{\ma}_x^{K,\mathrm{S}}(\alpha(a),\beta(b)) 
    \end{equation}
    for any two distance realizers as in \ref{def-cb-ang}(iii). 
    Clearly, any comparison angle is finite by definition (it is a hyperbolic angle in the model spaces between timelike distance realizers). If $\sigma=1$, i.e., if the two curves have different time orientation, then this becomes an inequality for non-signed (comparison) angles, and hence $\ma_x(\alpha,\beta) < \infty$ follows. For upper curvature bounds the inequality on signed angles is reversed, which is why we get the implication for finite $\ma_x(\alpha,\beta)$ for curves with the same time orientation (causing the inequality to reverse once again). 
    
    For (iii), we start out by noting that hinges and triangles are closely related concepts. Indeed, given any hinge $(\alpha, \beta)$ emanating from $x$ such that the endpoints $\alpha(a)$ and $\beta(b)$ of the curves are causally related, we can form a triangle $\Delta(x,\alpha(a), \beta(b))$ (the order of the points might change depending on the time orientation of the curves, and the side opposite of $x$ might be null, but this is not important for our arguments). 
    Conversely, any timelike triangle $\Delta(x,y,z)$ gives a hinge at $x$ (in fact, at any of the three points), by using the two sides adjacent to $x$.

    Say $(\alpha, \beta)$ is a hinge at $x$ with finite angle, both curves are future-directed, and $\alpha(a) \leq \beta(b)$. Consider the comparison triangle $\Delta(\bx, \bar{\alpha}(a), \bar{\beta}(b))$ (cf.\ the Realizabilty Lemma \cite[Lemma 2.1]{AB08})  and the comparison hinge $(\tilde{\alpha}(a), \tx, \tilde{\beta}(b))$. By construction, we have $\ma_x(\alpha, \beta)=\ma_{\tx}(\tilde{\alpha}(a), \tilde{\beta}(b))$ and $\tilde{\ma}_x(\alpha(a), \beta(b))=\ma_{\bx}(\bar{\alpha}(a), \bar{\beta}(b))$. 
    Moreover, the comparison hinge $(\tilde{\alpha}(a), \tx, \tilde{\beta}(b))$ can be viewed as a geodesic triangle, although the side connecting $\tilde{\alpha}(a)$ and $\tilde{\beta}(b))$ might not be causal. In any case, the sides adjacent to $\bx$ and $\tx$ have the same lengths, so we have $\ma_x(\alpha, \beta) \leq \tilde{\ma}_x(\alpha(a), \beta(b))$ if and only if $\tau(\alpha(a), \beta(b)) \geq \tau(\tilde{\alpha}(a), \tilde{\beta}(b))$ due\footnote{We prefer to use this result instead of our version of the law of cosines, as the triangle $\Delta(\tilde{x}, \tilde{\alpha}(a), \tilde{\beta}(b))$ has a possibly spacelike side.} to the Hinge Lemma \cite[Lemma 2.2]{AB08}. The case of $\alpha$ and $\beta$ having different time orientation (or both being past-directed) is completely analogous. 

    Finally, we need to touch on a small technicality about causal relations. While angle comparison talks about curves where the endpoints are causally related, this is not the case for hinge comparison, meaning that one needs to conclude from angle comparison the fact that hinge comparison is also valid in configurations where the endpoints are not causally related. Clearly, this is only possible if the curves have the same time-orientation, say both are future directed. Moreover, for upper curvature bounds, the inequality in hinge comparison reads $\tau(\alpha(a),\beta(b)) \leq \tau(\tilde{\alpha}(a),\tilde{\beta}(b))$, which is trivially satisfied if $\alpha(a)$ and $\beta(b)$ are not causally related. 
    So assume we are in the case of lower curvature bounds and
    let $(\alpha, \beta)$ be a hinge with both realizers future-directed and assume that there is no causal relation between $\alpha(a)$ and $\beta(b)$. We essentially need to show that there is no timelike relation between $\bar{\alpha}(a)$ and $\bar{\beta}(b)$. 
    Note that contrary to comparison triangles, comparison hinges have the useful property that sub-comparison hinges ``live inside'' the original one. 
    In other words, if $(\tilde{\alpha},\tilde{\beta})$ is a comparison hinge for $(\alpha, \beta)$, then $(\tilde{\alpha}|_{[0,s]},\tilde{\beta})$ is a comparison hinge for $(\alpha|_{[0,s]},\beta)$. 
    Clearly, $\alpha(\delta) \ll \beta(b)$ for small enough $\delta >0$. 
    Thus, together with $\tau(\alpha(a), \beta(b))=0$ and the mean value theorem ($\tau$ is continuous in a comparison neighbourhood), we infer that for each $\varepsilon>0$ small enough there is a parameter $s'$ such that $\tau(\alpha(s'), \beta(b))=\varepsilon$. Thus, we infer $\varepsilon=\tau(\alpha(s'),\beta(b)) \geq \tau(\tilde{\alpha}(s'),\tilde{\beta}(b))$. Since $\tau(\tilde{\alpha}(s), \tilde{\beta}(b))$ is clearly monotonically decreasing in $s$ as well and $\varepsilon>0$ was arbitrary, we arrive at $\tau(\tilde{\alpha}(a),\tilde{\beta}(b))=0$, as claimed. 
\end{proof}

\section{New characterizations of curvature bounds}
In this chapter we introduce several characterizations of curvature bounds which are new in the Lorentzian context.
\subsection{Timelike and causal curvature bounds}
Before we go on to introduce new formulations of curvature bounds, however, we want to briefly touch on the interplay between causal and timelike curvature bounds. Causal curvature bounds were also introduced in \cite[Definition 4.14]{KS18}. In (ii) of that definition it was required that comparison neighbourhoods be causally geodesic. However, since the defining inequalities on $\tau$ are only required between (comparison) points on timelike sides of an admissible causal triangle, the existence of null realizers is not necessary. 
For this reason, in the definition of causal curvature bounds we give below, we require comparison neighbourhoods to be merely $D_K$-geodesic (instead of causally $D_K$-geodesic), just as in the other formulations of curvature bounds. 
Note that with this modification all results about causal curvature bounds that have been obtained in the literature so far retain their validity. 
Most importantly, with this reformulation we are able to show that causal and timelike curvature bounds are in fact equivalent. 

\begin{defi}[(Strict) causal curvature bounds by triangle comparison]
Let $X$ be a \LpLSn. An open subset $U$ is called a $(\geq K)$- (resp.\ $(\leq K)$-)comparison neighbourhood in the sense of \emph{causal triangle comparison} if:
\begin{enumerate}
\item $\tau$ is continuous on $(U\times U)\cap\tau^{-1}([0,D_K))$, and this set is open.
\item $U$ is $D_K$-geodesic.
\item Let $\Delta(x,y,z)$ be an admissible causal triangle in $U$, with $p,q$ two points on the timelike sides of $\Delta(x,y,z)$. 
Let $\Delta(\bar{x},\bar{y},\bar{z})$ be a comparison triangle in $\lm{K}$ for  $\Delta(x,y,z)$ and $\bar{p},\bar{q}$ comparison points for $p$ and $q$, respectively. Then
\begin{equation}
\tau(p,q)\leq\tau(\bar{p},\bar{q})\quad \text{ (resp.\ } \tau(p,q) \geq \tau(\bar{p}, \bar{q}) \text{)} \, .
\end{equation}
\end{enumerate}
If in (iii) we additionally have 
\begin{equation}\label{eq:strict_causal_comp_implication}
p\leq q \Rightarrow \bar{p}\leq \bar{q}\quad \text{ (resp.\ } p\leq q \Leftarrow \bar{p}\leq \bar{q} \text{)} \, ,
\end{equation}
then $U$ is called a $(\geq K)$- (resp.\ $(\leq K)$-)comparison neighbourhood in the sense of \emph{strict causal triangle comparison}.
\end{defi}
\begin{thm}[Timelike and (strict) causal curvature bounds]
\label{thm:equivCausTl}
Let $X$ be a \LpLSn. 
\begin{itemize}
\item[(i)] Let $U\subseteq X$ be open. Then $U$ is a $(\geq K)$- (resp.\ $(\leq K)$-)comparison neighbourhood in the sense of timelike triangle comparison if and only if it is one in the sense of causal triangle comparison.
\item[(ii)] Let $U\subseteq X$ be open. 
%In the $(\leq K)$-case, additionally suppose that $U$ is causally closed. 
Then $U$ is a $(\geq K)$-comparison neighbourhood in the sense of causal triangle comparison if and only if it is one in the sense of strict causal triangle comparison. If $U$ is locally causally closed, the analogous statement about $U$ being a $(\leq K)$-comparison neighbourhood holds as well. 
\item[(iii)] $X$ has curvature bounded below (resp.\ above) by $K$ in the sense of timelike triangle comparison if and only if it has the same bound in the sense of causal triangle comparison.
\item[(iv)] $X$ has curvature bounded below by $K$ in the sense of causal triangle comparison if and only if it has the same bound in the sense of strict causal triangle comparison. If $X$ is strongly causal and locally causally closed, the analogous statement about $X$ having curvature bounded above by $K$ holds as well.
\end{itemize} 
\end{thm}

\begin{proof}
(i) For the non-trivial direction of the claim, suppose that $U$ is a comparison neighbourhood in the sense of timelike triangle comparison, and let $\Delta(x,y,z)$ be a causal triangle in $U$ satisfying size bounds, and with $\tau(y,z)=0$. Let $p$ and $q$ be points on the timelike sides of the triangle. 
Furthermore, let $\Delta(\bx,\by,\bz)$ be a comparison triangle for $\Delta(x,y,z)$ and denote by $\bp$ and $\bq$ the comparison points for $p$ and $q$ in that triangle, respectively. 
We distinguish the following cases:

1.) $p, q \neq y$: Let $p \in [x, y]$, $q \in [x, z]$, and let
$\alpha: [0,1]\to U$ be a geodesic realizing $[x, y]$, so $y=\alpha(1)$ and, say, $p=\alpha(t_0)$ for $t_0 \in (0,1)$. 
Set $y_t \coloneqq \alpha(t)$, then for  any $t\in (t_0,1)$, $\Delta(x,y_t,z)$ is a timelike triangle containing
$p$ and $q$. Let $\Delta(\bar x,\bar y_t,\bar z)$ be a comparison triangle for $\Delta(x,y_t,z)$ and denote by $\bp_t$ and $\bq_t$ the 
comparison points for $p$ and $q$ therein. 
Then by timelike triangle comparison for $\Delta(x,y_t,z)$, we have $\tau(p,q) \le \tau(\bar p_t, \bar q_t)$ 
(resp.\ $\tau(p,q) \ge \tau(\bar p_t, \bar q_t)$). 
We now argue that $\by_t \to \by$. 
Clearly, $y_t \to y$ and hence, since $\tau$ is continuous, $\tau(\bx,\by_t)=\tau(x,y_t) \to \tau(x,y)=\tau(\bx,\by)$ and $\tau(\by_t,\bz) = \tau(y_t,z) \to \tau(y,z)=\tau(\by,\bz)$. 
Fixing the segment $[\bx,\bz]$ in its place and assuming it is vertical (after applying a suitable Lorentz transformation), we see that $\by_t$ arises as the unique (up to reflection on $[\bx,\bz]$) point of intersection of hyperbolas with centers $\bx$ and $\bz$, respectively. 
Since $\tau$ is continuous, these hyperbolas transform continuously in $t$, with the one centered at $\bz$ degenerating into two line segments as $t \to 1$.
This shows $\by_t \to \by$, which immediately implies $\bp_t \to \bp$ and $\bq_t \to \bq$, and so $\tau(\bar p_t, \bar q_t) \to \tau(\bar p, \bar q)$. 
Thus we get $\tau(p,q)\le \tau(\bar p,\bar q)$ (resp.\ $\tau(p,q)\ge \tau(\bar p,\bar q)$), as claimed. 
The case of $p \in [x,z], q \in [x,y]$ is analogous. 

2.) $q=y$, $p\in [x, z]$: Again let $[x,y]$ be realized by the geodesic $\alpha: [0,1]\to U$ and set $q_t:=\alpha(t)$. 
By case 1.), $\tau(p,q_t) \le \tau(\bar p, \bar q_t)$ (resp.\ $\tau(p,q_t) \ge \tau(\bar p, \bar q_t)$), where $\bar q_t$ is the comparison point to $q$
in the triangle $\Delta(\bar x,\bar y,\bar z)$. Letting $t\nearrow 1$ we obtain $\tau(p,q)\le \tau(\bar p,\bar q)$ 
(resp.\ $\tau(p,q)\ge \tau(\bar p,\bar q)$) also in this case.

3.) $p=y$: Then $0=\tau(p,q)=\tau(\bar p, \bar q)$.

(ii) Let $U$ as in (ii) be a comparison neighbourhood in the sense of causal triangle comparison. Let $\Delta(x,y,z)$ be a causal triangle satisfying size bounds, and let $p,q$ each be either one of $x,y,z$ or lie on a timelike side of $\Delta(x,y,z)$. 
Let $\Delta(\bar{x},\bar{y},\bar{z})$ be a comparison triangle and $\bar{p},\bar{q}$ comparison points. 
Note that if $p,q$ are both vertices or lie on the same side, the required inequalities are always satisfied. So we may suppose that $p$ lies in the interior of the side $[x,z]$, say $p=\alpha(t_0)$, $t_0\in (0,1)$, where the geodesic $\alpha:[0,1]\to U$ realizes $[x,z]$. 
Let $p_t:= \alpha(t)$ for $t<t_0$. 
Then $p_t\ll p$ and $p_t\to p$ as $t\nearrow t_0$,
and similarly for the comparison points in $\Delta(\bar x,\bar y,\bar z)$ we have $\bar p_t\to \bar p$ as $t\nearrow t_0$. 
From this and the reverse triangle inequality for $\tau$ it follows that $p\le q \Rightarrow \tau(p_t,q)>0$ for all $t<t_0$. 
If $U$ is causally closed, the converse implication holds as well. 
Moreover, also in $\lm{K}$ we have $\tau(\bar p_t,\bar q) >0$ for all $t<t_0$ if and only if $\bar p \le \bar q$. 
Since $\tau(p_t,q) \le \tau(\bar p_t,\bar q)$ (resp.\ $\tau(p_t,q) \ge \tau(\bar p_t,\bar q)$), this verifies \eqref{eq:strict_causal_comp_implication} and thereby shows that $U$ is also a comparison neighbourhood in the sense of strict causal triangle comparison.

(iii) This is immediate from (i).

(iv) Recalling Lemma \ref{lem:causally_closed_comp_nbhds}, this is a direct consequence of (ii). 
\end{proof}

\begin{rem}[One-sided versions of (strict) causal triangle comparison] 
\label{rem: one sided strict causal}
In analogy to Definition \ref{def-cb-tr-onesided}, one can also introduce one-sided versions of (strict) causal triangle comparison by requiring one of $p,q$ to be a vertex of the triangle. 
The implication from (strict) causal triangle comparison to (strict) causal one-sided triangle comparison is obvious. 
The implications from strict causal one-sided triangle comparison to causal one-sided triangle comparison  and further to timelike one-sided triangle comparison are similiarly obvious. 
Thus, under the assumptions of Theorem \ref{thm:equivCausTl}, all of these notions are equivalent. 
\end{rem}

\subsection{The four-point condition}
In Alexandrov geometry, the four-point condition is a convenient reformulation used in both upper and lower curvature problems. Notably, it is used in a version of Toponogov's Theorem, cf. \cite{BGP92}. Its biggest advantage is that it does not require the existence of distance-realizers, i.e., it also works in a non-intrinsic setting. 
It is somewhat unique in the sense that, as the name suggest, it uses four points in contrast to essentially all previous formulations, which used three points (forming a hinge or a triangle), but at the same time the formulation is still fundamentally geometric in nature, so to say, in contrast to the convexity/concavity condition on $\tau$, which seems more analytical. 
The four-point condition is a bit more natural for the curvature bounded below case, which is why we give the definitions separately. As in the metric version, the four-point condition can be expressed both via distance and angle inequalities. 

Before giving the definition, it will be convenient to lay out some notational conventions.
\begin{defi}[Four-point configurations]\label{def:four-point configs}
    Let $X$ be a \LpLSn. 
    \begin{enumerate}
        \item By a \emph{timelike future four-point configuration} we mean a tuple of four points in $X$, usually denoted by $(y,x,z_1,z_2)$, satisfying the relations $y \ll x \ll z_1$ and $x \ll z_2$. 
        It is called \emph{endpoint-causal} if $z_1 \leq z_2$.

        \item By a \emph{causal future four-point configuration} we mean a tuple of four points in $X$, usually denoted by $(y,x,z_1,z_2)$, satisfying the relations $y \ll x \leq z_1$ and $x \leq z_2$. 
        It is called \emph{endpoint-causal} if $z_1 \leq z_2$.

        \item Given a timelike (resp.\ causal) future four-point configuration $(y,x,z_1,z_2)$ in $X$, by a \emph{four-point comparison configuration} in $\lm{K}$ we mean a tuple of four points $(\hy, \hx, \hz_1, \hz_2)$ such that $\tau(y,x)=\tau(\hy,\hx),\  \tau(y,z_i)=\tau(\hy,\hz_i)$ and $\tau(x,z_i)=\tau(\hx,\hz_i), i=1,2$, and such that $\hz_1$ and $\hz_2$ lie on opposite sides of the line through $\hy$ and $\hx$, see Figure \ref{fig: 4pt}. 
        
        \item A timelike (resp.\ causal) future four-point configuration in $X$ is called \emph{left (resp.\ right) straight} if $\tau(y,z_1)=\tau(y,x)+\tau(x,z_1)$ (resp.\ $\tau(y,z_2)=\tau(y,x)+\tau(x,z_2)$), i.e., $y, x$ and $z_1$ (resp.\ $z_2$) lie on a distance realizer, if it exists. 
        Note that the four-point comparison configuration (if it exists) of a four-point configuration is straight if and only if the original four-point configuration is straight, see Figure \ref{fig: straight 4pt}. 

        \item There are past versions of all of the aforementioned concepts, which result from reversing all causality relations in the obvious way. The resulting tuple will then be denoted by $(z_2, z_1, x, y)$. 
        
        When proving statements where some formulation of curvature bounds implies a curvature bound expressed via four-point configurations, we will only explicitly show how to obtain the desired inequality for a future configuration. 
        The case of a past configuration always follows symmetrically. 
        Since the list of decorating adjectives for four-point configurations is already quite long, we decided to omit the word `future' when dealing with future four-point configurations (which, in any case, are also clearly identified by the order of points in the above notation). 
    \end{enumerate}
\end{defi}
\begin{figure}
\begin{center}
\begin{tikzpicture}%[line cap=round,line join=round,>=triangle 45,x=1cm,y=1cm]
%yx
\draw (0,-1) .. controls (0.2,-0.5) and (-0.1, -0.25) .. (0,0);
%xz1
\draw (0,0) .. controls (-0.2,0.4) and (-0.3, 0.6) .. (-0.3,0.7);
%xz2
\draw (0,0) .. controls (0.4,0.75) and (0.3, 1) .. (0.3,1.5);
%z1z2
\draw[dashed] (-0.3,0.7) .. controls (-0.2,0.85) and (0.1, 1.2) .. (0.3,1.5);
%yz_1
\draw[] (0,-1) .. controls (-0.3,-0.5) and (-0.3, 0.25) .. (-0.3,0.7);
%yz_2
\draw[] (0,-1) .. controls (0.6,0.6) and (0.5, 1) .. (0.3,1.5);

\begin{scriptsize}
%real pts
\coordinate [circle, fill=black, inner sep=0.5pt, label=0: {$x$}] (ba) at (0,0);
\coordinate [circle, fill=black, inner sep=0.5pt, label=270: {$y$}] (bb) at (0,-1);
\coordinate [circle, fill=black, inner sep=0.5pt, label=180: {$z_1$}] (bx) at (-0.3,0.7);
\coordinate [circle, fill=black, inner sep=0.5pt, label=0: {$z_2$}] (by) at (0.3,1.5);
%comp pts
\coordinate [circle, fill=black, inner sep=0.5pt, label=90: {$\hx$}] (hx) at (3,0);
\coordinate [circle, fill=black, inner sep=0.5pt, label=270: {$\hy$}] (hy) at (3,-1);
\coordinate [circle, fill=black, inner sep=0.5pt, label=180: {$\hz_1$}] (hz1) at (2.5,0.6);
\coordinate [circle, fill=black, inner sep=0.5pt, label=0: {$\hz_2$}] (hz2) at (3.7,1.6);
\end{scriptsize}
%comp sit
\draw (hy) -- (hx) -- (hz1);
\draw (hx) -- (hz2);
\draw[dashed] (hz1) -- (hz2);
\draw[] (hy) -- (hz1);
\draw[] (hy) -- (hz2);
\end{tikzpicture}
\caption{A timelike four-point configuration in $X$ and a corresponding comparison configuration. }
\label{fig: 4pt}
\end{center}
\end{figure}
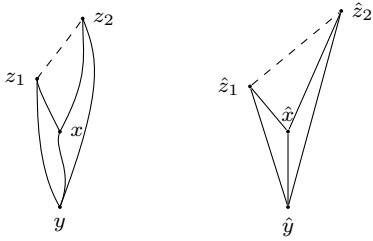

\begin{figure}
\begin{center}
\begin{tikzpicture}%[line cap=round,line join=round,>=triangle 45,x=1cm,y=1cm]
%yxz1
\begin{scriptsize}
\draw (0,-1) .. controls (-0.3,-0.25) and (0,0.5) .. (-0.3,0.7) node (A)[circle, fill=black,inner sep=0.5pt,pos=0.4,label=0:$x$]{};
\end{scriptsize}
%xz2
\draw (A) .. controls (0.2,0.75) and (0.3, 1) .. (0.3,1.5);
%z1z2
\draw[dashed] (-0.3,0.7) .. controls (-0.2,0.85) and (0.1, 1.2) .. (0.3,1.5);
%yz_2
\draw[] (0,-1) .. controls (0.6,0.6) and (0.5, 1) .. (0.3,1.5);

%comp sit
\draw (3,-1) -- (3,0) -- (3,0.7);
\draw (3,0) -- (3.7,1.5);
\draw[dashed] (3,0.7) -- (3.7,1.5);
\draw[] (3,-1) -- (3.7,1.5);
\begin{scriptsize}
%real pts
\coordinate [circle, fill=black, inner sep=0.5pt, label=270: {$y$}] (bb) at (0,-1);
\coordinate [circle, fill=black, inner sep=0.5pt, label=180: {$z_1$}] (bx) at (-0.3,0.7);
\coordinate [circle, fill=black, inner sep=0.5pt, label=0: {$z_2$}] (by) at (0.3,1.5);
%comp pts
\coordinate [circle, fill=black, inner sep=0.5pt, label=180: {$\hx$}] (ba) at (3,0);
\coordinate [circle, fill=black, inner sep=0.5pt, label=270: {$\hy$}] (bb) at (3,-1);
\coordinate [circle, fill=black, inner sep=0.5pt, label=180: {$\hz_1$}] (bx) at (3,0.7);
\coordinate [circle, fill=black, inner sep=0.5pt, label=0: {$\hz_2$}] (by) at (3.7,1.5);
\end{scriptsize}

\end{tikzpicture}
\caption{A (left) straight timelike four-point configuration in $X$ and a corresponding comparison configuration.}
\label{fig: straight 4pt}
\end{center}
\end{figure}
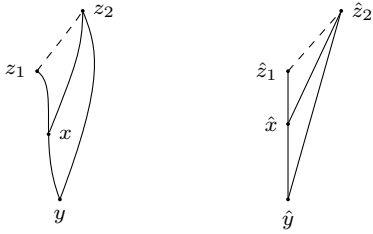
Intuitively, these four-point configurations could be thought of as two admissible causal (or even timelike) triangles $\Delta(y,x,z_i)$ that share the side $[y,x]$, but technically one has to be careful with this as the points in $X$ might not form a triangle if there are no geodesics joining the points. 
\begin{defi}[Size bounds for four-point configurations]
    Similar to the corresponding terminology for triangles and hinges, a four-point configuration $(y, x, z_1, z_2)$ in a \LpLS $X$ is said to \emph{satisfy size bounds for $K$} if there exists a four-point comparison configuration in $\lm{K}$. Evidently, this is the case precisely if $\tau(y,z_1)<D_K$ and $\tau(y,z_2)<D_K$. Note that the four-point comparison configuration is unique up to isometry of $\lm{K}$. 
    Throughout this work, we will assume all mentioned four-point configurations satisfy size bounds. 
\end{defi}
It turns out that lower and upper curvature bounds in the sense of any four-point condition have to be formulated quite differently, which is why we introduce them separately. We will go into more detail below.

\begin{defi}[Lower curvature bounds by timelike (resp.\ causal) four-point condition]
\label{def-cb-4pt-cbb}
Let $X$ be a \LpLSn. An open subset $U$ of $X$ is called a $(\geq K)$-comparison neighbourhood in the sense of the \emph{timelike (resp.\ causal) four-point condition} if:
\begin{enumerate}
\item $\tau$ is continuous on $(U\times U) \cap \tau^{-1}([0,D_K))$, and this set is open.
\item \label{def-cb-4pt-cbb.main} Let $(y, x, z_1, z_2)$ be a timelike (resp.\ causal) and endpoint-causal four-point configuration in $U$. Let $(\hy, \hx, \hz_1, \hz_2)$ be a four-point comparison configuration in $\lm{K}$. Then 
\begin{equation}
\tau(z_1,z_2) \geq \tau(\hz_1,\hz_2) \, .    
\end{equation}
In addition, for any timelike (resp.\ causal) and endpoint-causal past four-point configuration $(z_2, z_1, x, y)$ and a comparison configuration $(\hz_2, \hz_1, \hx, \hy)$, we require
\begin{equation}
\tau(z_2,z_1) \geq \tau(\hz_2,\hz_1) \, .    
\end{equation}
\end{enumerate}
\end{defi}
In the spirit of the equivalence between timelike and causal curvature bounds established in Theorem \ref{thm:equivCausTl}, we also give a more general version of the four-point condition. 
This is a priori a stricter property as it says more about a greater number of configurations. 
Their equivalence (under some mild assumptions) will be demonstrated below. 
Unsurprisingly, it is very convenient to have different equivalent formulations of the same property at hand. 
In particular, we expect the strict causal triangle comparison and the strict causal four-point condition to be especially useful for the slightly adapted setting of so-called \emph{Lorentzian metric spaces}, cf.\ \cite{McC23}, where the time separation function $\tau$ is replaced by a function $\ell$ that additionally encodes the causal relation. 
These conditions can also be more concisely formulated in terms of $\ell$. 

\begin{defi}[Lower curvature bounds by strict causal four-point condition]
\label{def:general_4_point}
Let $X$ be a \LpLSn. An open subset $U$ in $X$ is called a $(\geq K)$-comparison neighbourhood in the sense of the \emph{strict causal four-point condition} if it is a $(\geq K)$-comparison neighbourhood in the sense of the causal four-point condition, where condition (ii) in Definition \ref{def-cb-4pt-cbb} is strengthened to:

\begin{enumerate}
\item[(ii')] 
Let $(y, x, z_1, z_2)$ be a causal four-point configuration in $U$ (not necessarily endpoint-causal). 
Let $(\hy, \hx, \hz_1, \hz_2)$ be a four-point comparison configuration in $\lm{K}$. Then 
\begin{equation}
\tau(z_1,z_2) \geq \tau(\hz_1,\hz_2) \, \text{and}    
\end{equation}
\begin{equation}
\hat{z}_1\leq\hat{z}_2 \Rightarrow z_1\leq z_2\, .
\end{equation}
In addition, for any causal past four-point configuration $(z_2, z_1, x, y)$ and a comparison configuration $(\hz_2, \hz_1, \hx, \hy)$ we require 
\begin{equation}
\tau(z_2,z_1) \geq \tau(\hz_2,\hz_1) \, \text{and}    
\end{equation}
\begin{equation}
\hat{z}_2 \leq \hat{z}_1 \Rightarrow z_2 \leq z_1 \, .
\end{equation}
\end{enumerate}
\end{defi} 
Similar to the metric case, the timelike four-point condition can not only be described via a distance estimate but also via the behaviour of (comparison) angles, cf.\ \cite[Definition 2.3]{BGP92} and \cite[Definition II.1.10]{BH99}. 
\begin{lem}[Angle version of the timelike four-point condition for lower curvature bounds]
\label{lem-angle version 4pt}
Let $U$ be an open subset in a Lorentzian pre-length space $X$ which satisfies Definition \ref{def-cb-4pt-cbb}(i). 
Let $(y, x, z_1, z_2)$ be a timelike and endpoint-causal four-point configuration in $U$ and let $(\hy, \hx, \hz_1, \hz_2)$ be a comparison configuration in $\lm{K}$. Then $\tau(z_1,z_2) \geq \tau(\hz_1,\hz_2)$ if and only if
\begin{equation}
\label{eq: angle version four point}
\tilde{\ma}_x(z_1,z_2)\leq\tilde{\ma}_x(z_1,y)+\tilde{\ma}_x(y,z_2)\,.
\end{equation}
In addition, if $(z_2,z_1,x,y)$ is a timelike past and endpoint-causal four-point configuration and $(\hz_2, \hz_1, \hx, \hy)$ a comparison configuration, then $\tau(z_2,z_1) \geq \tau (\hz_2, \hz_1)$ if and only if \eqref{eq: angle version four point} is satisfied.
\end{lem}
\begin{proof}
It will suffice to only consider the future case. 
Let $\Delta(\bx,\bz_1,\bz_2)$ be a comparison triangle for (the possibly merely causal) triangle $\Delta(x,z_1,z_2)$, then by definition $\tilde{\ma}_x(z_1,z_2) = \ma_{\bx}(\bz_1,\bz_2)$. 
For the hyperbolic angles in the comparison configuration $(\hy, \hx, \hz_1, \hz_2)$, we have
\begin{equation}
\ma_{\hat{x}}(\hat{z}_1,\hat{z}_2)=\ma_{\hat{x}}(\hat{y},\hat{z}_1) + \ma_{\hat{x}}(\hat{y},\hat{z}_2)  = \tilde{\ma}_x(z_1,y)+\tilde{\ma}_x(y,z_2) \, , 
\end{equation}
where the first equality is due to the triangle equality of angles in (2-dimen\-sio\-nal) spacetimes. Note that $\tau(\bar{x},\bar{z}_1)=\tau(x,z_1)=\tau(\hat{x},\hat{z}_1)$ and $\tau(\bar{x},\bar{z}_2)=\tau(x,z_2)=\tau(\hat{x},\hat{z}_2)$, i.e., the sides adjacent to the angles at $\bx$ and $\hx$ have the same lengths. Thus, we can use the Hinge Lemma, cf.\ \cite[Lemma 2.2]{AB08}, to read off the desired equivalence directly. 
\end{proof}

\begin{pop}[Angle comparison implies timelike four-point condition for lower curvature bounds]
\label{pop: angle comparison implies four-point condition CBB}
Let $U$ be an open subset in a regular Lorentzian pre-length space $X$. 
If $U$ is a $(\geq K)$-comparison neighbourhood in the sense of angle comparison, then it is a $(\geq K)$-comparison neighbourhood in the sense of the timelike four-point condition.
\end{pop}
\begin{proof}
Let $U$ be a $(\geq K)$-comparison neighbourhood in the sense of angle comparison and let $(y, x, z_1, z_2)$ be a timelike and endpoint-causal four-point configuration in $U$. Take distance realizers $\alpha$ from $x$ to $z_1$, $\beta$ from $x$ to $y$, $\gamma$ from $x$ to $z_2$, which we know exist by Definition \ref{def-cb-ang}(ii). 
We obtain the following inequalities for angles:
\begin{equation}
    \tilde{\ma}_x(z_1,z_2)\leq\ma_x(\alpha,\gamma)\leq\ma_x(\alpha,\beta)+\ma_x(\beta,\gamma) \leq \tilde{\ma}_x(z_1,y)+\tilde{\ma}_x(y,z_2), 
\label{eq: lcb-angle-->4pt-1}
\end{equation}
where we used (iii) (with corresponding signs) and (iv) in Definition \ref{def-cb-ang}. 
The claim therefore follows from Lemma \ref{lem-angle version 4pt}. 
The same inequality can be obtained for a past four-point configuration in complete analogy. 
\end{proof}

When introducing a four-point condition for timelike curvature bounded above, we run into the following problem: in the above proof, the inequalities from angle comparison and hinge comparison reverse, but the triangle inequality of angles does not. 
However, there is a way around this, as equality in the triangle inequality of angles is enough to obtain inequalities in the opposite direction. 
This is achieved by restricting to straight four-point configurations. 
We also briefly want to justify (ii) in the following definition. 
In essence, the existence of $\tau$-midpoints is assumed\footnote{The existence of $\tau$-midpoints is comparatively strong, but it is easy to formulate and we are mostly working in an intrinsic setting anyways (where the existence of such points is not automatic). 
Technically, requiring a weaker condition, like the existence of distance realizers that are partially defined on a dense subset of some interval suffices.} in order to ensure that the definition does turn into a void statement when its assumptions cannot be met. 
Indeed, there exist exotic spaces without distance realizers, where there simply exist too few (or none at all) straight four-point configurations, in which case the curvature bound might be trivially satisfied. 
As an example for such a space, consider a locally finite random selection of points in the Minkowski plane, equipped with the restrictions of the causal relation and time separation function from ambient space. 
Then almost surely no three points lie on a line.
\begin{defi}[Upper curvature bounds by timelike (resp.\ causal) four-point condition]
\label{def-cb-4pt-cba}
Let $X$ be a \LpLSn. 
An open subset $U$ is called a $(\leq K)$-comparison neighbourhood in the sense of the \emph{timelike (resp.\ causal) four-point condition} if:
\begin{enumerate}
\item $\tau$ is continuous on $(U\times U) \cap \tau^{-1}([0,D_K))$, and this set is open. 

\item For all $x \ll z$ in $U$ with $\tau(x,z) < D_K$ there exists a $\tau$-midpoint in $U$, i.e., a point $y \in U$ such that $\tau(x,y)=\tau(y,z)=\frac{1}{2}\tau(x,z)$. 

\item Let $(y, x, z_1, z_2)$ be a straight timelike (resp.\ causal) and endpoint-causal four-point configuration in $U$. 
Let $(\hat{y}, \hat{x}, \hat{z}_1 , \hat{z}_2)$ be a straight four-point comparison configuration in $\lm{K}$. 
Then 
\begin{equation}
\tau(z_1,z_2) \leq \tau(\hat{z}_1,\hat{z}_2)\,.
\end{equation} 
In addition, for any straight timelike (resp.\ causal) and endpoint-causal past four-point configuration $(z_2, z_1, x, y)$ and a comparison configuration $(\hz_2, \hz_1, \hx, \hy)$ we require 
\begin{equation}
\tau(z_2,z_1) \leq \tau(\hat{z}_2,\hat{z}_1)\,.
\end{equation} 
\end{enumerate}
\end{defi}
\begin{defi}[Upper curvature bounds by strict causal four-point condition]
Let $X$ be a \LpLSn. An open subset $U$ of $X$ is called a $(\leq K)$-comparison neighbourhood in the sense of the \emph{strict causal four-point condition} if it is a $(\leq K)$-comparison neighbourhood in the sense of the causal four-point condition, where under the assumptions of condition (iii) we additionally require 
\begin{equation}\label{eq:causality_implication_strict_4pt}
    z_1\leq z_2\Rightarrow\hat{z}_1\leq\hat{z}_2 \, ,
\end{equation}
and for past configurations we additionally require 
\begin{equation}
    z_2\leq z_1 \Rightarrow \hat{z}_2 \leq \hat{z}_1 \, .
\end{equation}
\end{defi}
\begin{rem}[Endpoint-causality in the strict four-point condition]
\label{rem:endpoint-causality in strict 4pt}
Note that for upper curvature bounds, a more general formulation allowing for non-endpoint-causal four-point configurations is superfluous. Indeed, whenever $z_1$ and $z_2$ are not causally related, both the inequality $\tau(z_1,z_2) \leq \tau(\hz_1,\hz_2)$ and the implication $z_1\leq z_2\Rightarrow\hat{z}_1\leq\hat{z}_2$ are trivially satisfied. 
It therefore essentially only makes sense to consider endpoint-causal four-point configurations. 
\end{rem}
\begin{rem}[Relevant constellations of causal four-point configurations]
\label{rem: relevant constellations of four-point configs}
Here we show that it is not necessary to look at causal four-point configurations $(y,x,z_1,z_2)$ where $x \leq z_2$ are null related, where for curvature bounded below in the sense of the strict causal four-point condition one additionally needs that the space (or the comparison neighbourhood) is regular to conclude this. 
However, the latter is not an actual restriction since the only statement involving the strict causal four-point condition, Proposition \ref{prop:4-point_causal_noncausal_equivalent}, assumes this anyways. 

To begin with, we cannot have $z_1\ll z_2$ as otherwise $x\leq z_1\ll z_2$ would yield a timelike relation $x\ll z_2$, and the same works for $\hat{z}_1\not\ll\hat{z}_2$. 
In particular, the $\tau$-inequality in any four-point condition is trivially satisfied. 
Moreover, $\hz_1\not\leq\hz_2$ unless the four-point situation is left-straight and $x$ is null before $z_1$ (as $\hz_1$ and $\hz_2$ are on opposite sides of the line extending $[\hy,\hx]$). 
Under the previously mentioned assumption of regularity, $x=z_1$ follows in this case (as $y,x,z_1$ are collinear with $y\ll x$ and $x,z_1$ null related) and therefore lower curvature bounds in the sense of the strict causal four-point condition automatically hold. 

For upper curvature bounds in the sense of the strict causal four-point condition, assume $(y,x,z_1,z_2)$ is straight with $x \leq z_2$ null related. Since we need to consider endpoint-causal configurations by Remark \ref{rem:endpoint-causality in strict 4pt}, it must be the case that $x \leq z_1$ are null related as well, otherwise $x \ll z_1 \leq z_2$ would yield $x \ll z_2$, a contradiction to them being null related. 
If $(y,x,z_1,z_2)$ is left straight, we have $\tau(y,z_1)=\tau(y,x)+\tau(x,z_1) = \tau(y,x)$. 
In particular, the comparison points $\hat y,\hat x,\hat z_1$ lie on a line, and as $\tau(\hat x,\hat z_1)=0$, we conclude $\hz_1=\hx \leq \hz_2$.
Thus, $\hz_1=\hx \leq \hz_2$. 
On the other hand, if it is right straight, then $\tau(y,z_2)=\tau(y,x)+\tau(x,z_2)=\tau(y,x)$. 
Further, $\tau(y,z_2) \geq \tau(y,z_1) + \tau(z_1,z_2) = \tau(y,z_1)$ and also $\tau(y,z_1) \geq \tau(y,x) + \tau(x,z_1) = \tau(y,x)$, so $\tau(y,x) = \tau(y,z_1) = \tau(y,z_2)$. 
In particular, this configuration is also left straight, making both $\hat y,\hat x, \hat z_1$ and $\hat y,\hat x, \hat z_2$ lie on a line, which forces $\hx=\hz_1=\hz_2$. 
This shows that \eqref{eq:causality_implication_strict_4pt} is satisfied. 

Finally, if $x=z_1$ or $x=z_2$, any $\tau$-inequality and implication of causal relation is trivially satisfied. 

Altogether (assuming the space is regular in the case of the strict causal four-point condition for curvature bounded below), we can always assume that all four points are distinct and $x\ll z_2$. 
\end{rem}
In order to show that angle comparison implies timelike four-point comparison in the case of upper curvature bounds, we require the following auxiliary result. 
It is in fact a variant of \cite[Theorem 4.5(i)]{BS22}, where we do not rely on the fact that one of the angles exists, cf. \cite[Lemma 4.10]{BS22}. 
\begin{lem}[Triangle inequality of angles, special case]
\label{lem: triangle inequality special case}
    Let $X$ be a \LpLS with curvature bounded above by $K$ in the sense of timelike triangle comparison. Let $\alpha$ and $\beta$ be future directed distance realizers, and let $\gamma$ be a past-directed distance realizer, all emanating from the same point $p$, such that the concatenation of $\gamma$ and $\beta$ again is a distance realizer. Then 
    \begin{equation}
        \ma_p(\alpha, \gamma) \leq \ma_p(\alpha, \beta) \, . 
    \end{equation}
    Since $\ma_p(\beta, \gamma)=0$, cf.\ \cite[Lemma 3.4]{BS22}, this amounts to the following triangle inequality of angles:
    \begin{equation}
        \ma_p(\alpha, \gamma) \leq \ma_p(\alpha, \beta) + \ma_p(\beta, \gamma) \, . 
    \end{equation}
\end{lem}
\begin{proof}
    Choose any parameters $r,s,t$ such that, say, $x=\gamma(r), y=\beta(s)$ and  $z=\alpha(t)$ form a timelike triangle $\Delta(x,y,z)$ (the direction of the timelike relation between the points on $\alpha$ and $\beta$ is not important). 
    Consider the two subtriangles $\Delta(x,p,z)$ and $\Delta(p,y,z)$ and consider a comparison configuration consisting of $\Delta(\bx,\bp,\bz)$ 
    and $\Delta(\bp,\by,\bz)$ (such that they share the common side between $\bp$ and $\bz$). 
    Due to upper curvature bounds and the Alexandrov Lemma, cf.\  \cite[Proposition 2.42]{BORS23}, this is a concave configuration, i.e., 
    \begin{equation}
        \tilde{\ma}_p(\gamma(r),\alpha(t))=\ma_{\bp}(\bx, \bz) \leq \ma_{\bp}(\by, \bz)=\tilde{\ma}_p(\beta(s),\alpha(t))\,.
    \end{equation}
    The desired inequality then follows from the definition of angles as limits of comparison angles. 
\end{proof}
\begin{pop}[Angle comparison implies timelike four-point condition for upper curvature bounds]
\label{pop: angle comparison implies four-point condition CBA}
Let $U$ be an open subset of a regular \LpLS $X$. 
If $U$ is a $(\leq K)$-comparison neighbourhood in the sense of angle comparison, then it is also a $(\leq K)$-comparison neighbourhood in the sense of the timelike four-point condition. 
\end{pop}
\begin{proof}
Properties (i) and (ii) in Definition \ref{def-cb-4pt-cba} follow directly from (i) and (ii) in Definition \ref{def-cb-ang}. 

So let $(y, x, z_1, z_2)$ in $U$ be a straight timelike and endpoint-causal four-point configuration. 
Take distance realizers (which exist by our assumptions in Definition \ref{def-cb-ang}) $\alpha$ from $x$ to $z_1$, $\beta$ from $x$ to $y$, and $\gamma$ from $x$ to $z_2$. 
Note that for, say, a left straight configuration, $\alpha$ and $\beta$ fit together to a distance realizer from $y$ through $x$ to $z_1$ (the right straight case works analogously, with $\beta,\gamma$ fitting together).  
In particular, $\ma_x(\alpha,\beta)=0$ (by \cite[Lemma 3.4]{BS22}). 
Let $(\hy, \hx, \hz_1, \hz_2)$ be a (straight) comparison configuration for $(y, x, z_1, z_2)$. 
In particular, $\Delta(\hy,\hx,\hz_2)$ is a comparison triangle for $\Delta(y,x,z_2)$.
Similar to the lower curvature bounds case, we obtain the following inequality for angles: 
\begin{equation}\label{eq:angle_to_straight_hinge}
\begin{split}
    \ma_x(\alpha,\gamma) & = \underbrace{\ma_x(\beta,\alpha)}_{=0} +\ma_x(\alpha,\gamma)\geq \ma_x(\beta,\gamma) \\ 
    & \geq \tilde{\ma}_x(y,z_2) = \ma_{\hx}(\hy,\hz_2) = \ma_{\hx}(\hz_1,\hz_2) \, , 
\end{split}
\end{equation}
where we used Lemma \ref{lem: triangle inequality special case}, Definition \ref{def-cb-ang}(iii) (with the sign of the angles already taken into account), and the triangle equality for angles in $\lm{K}$.

Let $(\tilde{x},\tilde{z}_1,\tilde{z}_2)$ form a comparison hinge for $(\alpha,\gamma)$ in $\lm{K}$. 
Then hinge comparison, cf.\ Definition \ref{def-cb-hinge} and Proposition \ref{prop:equivalence_hinge_angle} yield 
\begin{equation}
\label{eq:tau_in_angle_to_straight}
\tau(z_1,z_2) \leq \tau(\tilde{z}_1,\tilde{z}_2) \, .
\end{equation}
The comparison hinge $(\tilde{x},\tilde{z}_1,\tilde{z}_2)$ and the triangle $\Delta(\hx,\hz_1,\hz_2)$ have two sides of equal length, and $\ma_{\hat{x}}(\hat{z}_1,\hat{z}_2)\leq \ma_x(\alpha,\gamma)=\ma_{\tilde{x}}(\tilde{z}_1,\tilde{z}_2)$ by \eqref{eq:angle_to_straight_hinge}. 
Thus, law of cosines monotonicity (cf.\ \cite[Remark 2.5]{BS22}) implies $\tau(\tilde{z}_1,\tilde{z}_2) \leq \tau(\hat{z}_1,\hat{z}_2)$, which together with \eqref{eq:tau_in_angle_to_straight} gives the desired inequality $\tau(z_1,z_2) \leq \tau(\hat{z}_1,\hat{z}_2)$. 
The case of a past four-point configuration follows analogously. 
\end{proof}
\begin{pop}[Angle version of the timelike four-point condition for upper curvature bounds]
Let $U$ be an open subset in a Lorentzian pre-length space $X$ which satisfies Definition \ref{def-cb-4pt-cbb}(i). 
Let $(y, x, z_1, z_2)$ be a straight timelike and endpoint-causal four-point configuration in $U$ and let $(\hy, \hx, \hz_1, \hz_2)$ be a comparison configuration in $\lm{K}$. Then $\tau(z_1,z_2) \leq \tau(\hz_1,\hz_2)$ if and only if
\begin{equation}
\label{eq: angle version four point upper bounds}
\tilde{\ma}_x(z_1,z_2) \geq \tilde{\ma}_x(z_1,y) + \tilde{\ma}_x(y,z_2) \, .
\end{equation}
In addition, if $(z_2,z_1,x,y)$ is a timelike and endpoint-causal past four-point configuration and $(\hz_2, \hz_1, \hx, \hy)$ a comparison configuration, then $\tau(z_2,z_1) \leq \tau (\hz_2, \hz_1)$ if and only if \eqref{eq: angle version four point upper bounds} is satisfied.
\end{pop} 
\begin{proof}
The proof is completely analogous to the lower curvature bounds version, see Lemma \ref{lem-angle version 4pt}. 
\end{proof}
Note that in \eqref{eq: angle version four point upper bounds}, one of the angles on the right hand side is zero, depending on whether one deals with a left straight or a right straight configuration. 
\begin{pop}[Timelike vs.\ causal four-point condition]
\label{pop: timelike vs causal 4pt condition}
Let $X$ be a \LpLSn, and let $U \subseteq X$ be open, regular and $D_K$-geodesic. Then $U$ is a $(\geq K)$- (resp.\ $(\leq K)$-)comparison neighbourhood in the sense of the timelike four-point condition if and only if $U$ is a $(\geq K)$- (resp.\ $(\leq K)$-)comparison neighbourhood in the sense of the causal four-point condition. 

In particular, if $X$ is strongly causal, locally $D_K$-geodesic and regular, then it has curvature bounded below (resp.\ above) by $K$ in the sense of the timelike four-point condition if and only if it has the same bound in the sense of the causal four-point condition. 
\end{pop}
\begin{proof}
The direction from causal to timelike is clear, as any (straight) timelike four-point configuration is also a (straight) causal four-point configuration. 

For the converse direction, let $(y, x, z_1, z_2)$ be a causal and endpoint-causal four-point configuration. 
Let $\alpha:[0,1]\to X$ be the timelike distance realizer from $y$ to $x$. 
Set $x^t \coloneqq \alpha(t)$, then for all $t<1$ the four-point configuration $(y, x^t, z_1, z_2)$ is timelike and endpoint-causal, and straight if $(y, x, z_1, z_2)$ was straight. 
Note that by continuity of $\tau$ we can choose the four-point comparison configuration $(\hy^t, \hx^t, \hz_1^t, \hz_2^t)$ of $(y,x^t,z_1,z_2)$ such that each of the points converges to the corresponding point in the four-point comparison situation $(\hy, \hx, \hz_1, \hz_2)$ of $(y, x, z_1, z_2)$. 
In particular, we have $\tau(\hz_1^t,\hz_2^t)\to\tau(\hz_1,\hz_2)$, and $\tau(z_1,z_2)$ remains independent of $t$. 
For lower curvature bounds we know $\tau(z_1,z_2)\geq \tau(\hz_1^t,\hz_2^t)$, thus we also have $\tau(z_1,z_2)\geq \tau(\hz_1,\hz_2)$. 
In the case of upper curvature bounds, we get corresponding inequalities in the other direction. 
The case of a past four-point configuration follows analogously.

Finally, note that the additional assumptions in the second part of the claim are required since comparison neighbourhoods in the sense of any four-point condition need not be $D_K$-geodesic. Concerning the non-trivial direction, let $x \in X$ and suppose that $U$ is a comparison neighbourhood of $x$ in the sense of the timelike four-point condition. Then we find a neighbourhood $V$ of $x$ which is $D_K$-geodesic. Any intersection of timelike diamonds inside $U \cap V$ is, due to causal convexity, easily seen to be a regular and $D_K$-geodesic comparison neighbourhood, hence the first statement of the proposition applies. 
\end{proof}

Next, we show that curvature bounds in the sense of the causal four-point condition imply curvature bounds in the sense of monotonicity comparison. 
\begin{pop}[Causal four-point condition implies monotonicity comparison]
\label{pop: four point condition implies monotonicity comparison} 
Let $X$ be a \LpLS and let $U \subseteq X$ be open, regular and $D_K$-geodesic. 
If $U$ is a $(\geq K)$- (resp.\ $(\leq K)$-)comparison neighbourhood in the sense of the causal four-point condition, then $U$ is a $(\geq K)$- (resp.\ $(\leq K)$-)comparison neighbourhood in the sense of monotonicity comparison. 

In particular, if $X$ is strongly causal, regular and locally $D_K$-geodesic, and $X$ has curvature bounded below (resp.\ above) by $K$ in the sense of the causal four-point condition, then it also has the same bound in the sense of monotonicity comparison. 
\end{pop}
\begin{proof}
We only demonstrate the case of lower curvature bounds, the upper curvature bounds case is entirely analogous. 
Let $U$ be as in the statement. 
The first two conditions in Definition \ref{def-cb-mon} are satisfied by assumption.  
For the third condition, let $\alpha:[0,a]\to X$, $\beta:[0,b]\to X$ be a hinge with $\alpha(0)=\beta(0)$. 

There are two cases to consider, one where the two curves have the same time-orientation (say both future-directed), and one where they have different time-orientation (say $\alpha$ is future-directed and $\beta$ is past-directed). 

First, we consider the case of $\alpha$ and $\beta$ being future-directed. We need to show that the partial function $\theta(s,t)=\tilde{\ma}_y^{K,\mathrm{S}}(\alpha(s),\beta(t))$ is monotonically increasing. By Remark \ref{rem-cb-one-sided-mon}, it suffices to establish one-sided monotonicity. 
The future-directed case technically breaks down into three subcases, depending on the relations between the points on the curves (see Figure \ref{fig: 4pt-->mon} for a rough sketch of the in total four subcases). 
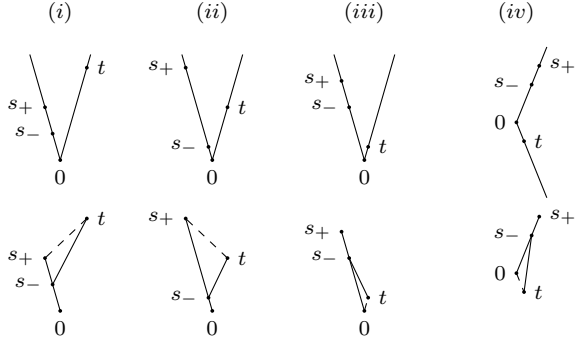
\begin{figure}
\begin{center}
\begin{tikzpicture}
%four hinges
\draw (-0.4,1.4) -- (0,0) -- (0.4,1.4);
\draw (-0.4+2,1.4) -- (0+2,0) -- (0.4+2,1.4);
\draw (-0.4+4,1.4) -- (0+4,0) -- (0.4+4,1.4);
\draw (0.4+6,-1+0.5) -- (0+6,0+0.5) -- (0.4+6,1+0.5);

%four 4pts
%4pt1
\draw (0,0-2) -- (-0.2,1.4*2/4-2) -- (-0.2,1.4*2/4-2); 
\draw (-0.1,1.4*1/4-2) -- (0.35,1.4*3.5/4-2);
\draw[dashed] (0.35,1.4*3.5/4-2) -- (-0.2,1.4*2/4-2);
%\draw (0,0-2) -- (0.35,1.4*3.5/4-2);

%4pt2
\draw (0+2,0-2) -- (-0.1+2,1.4*1/4-2) -- (-0.35+2,1.4*3.5/4-2);
\draw (-0.05+2,1.4*0.5/4-2) -- (0.2+2,1.4*2/4-2);
\draw[dashed] (0.2+2,1.4*2/4-2) -- (-0.35+2,1.4*3.5/4-2);
%\draw (0+2,0-2) -- (0.2+2,1.4*2/4-2);

%4pt3
\draw (0+4,0-2) -- (-0.3+4,1.4*3/4-2); 
\draw (-0.2+4,1.4*2/4-2) -- (0.05+4,1.4*0.5/4-2);
\draw[dashed] (0+4,0-2) -- (0.05+4,1.4*0.5/4-2);
%\draw (-0.3+4,1.4*3/4-2) -- (0.05+4,1.4*0.5/4-2);

%4pt4
\draw (0+6,0+0.5-2) -- (0.3+6,1*3/4+0.5-2);
\draw (0.2+6,1.*2/4+0.5-2) -- (0.1+6,-1*1/4+0.5-2);
\draw[dashed] (0+6,0+0.5-2) -- (0.1+6,-1*1/4+0.5-2);
%\draw (0.3+6,1*3/4+0.5-2) -- (0.1+6,-1*1/4+0.5-2);

\begin{scriptsize}
%hinge1
\coordinate [circle, fill=black, inner sep=0.5pt, label=270: {$0$}] (0.1) at (0,0); 
\coordinate [circle, fill=black, inner sep=0.5pt, label=180: {$s_+$}] (s+.1) at (-0.2,1.4*2/4);
\coordinate [circle, fill=black, inner sep=0.5pt, label=180: {$s_-$}] (s-.1) at (-0.1,1.4*1/4);
\coordinate [circle, fill=black, inner sep=0.5pt, label=0: {$t$}] (t.1) at (0.35,1.4*3.5/4);

%4pt1
\coordinate [circle, fill=black, inner sep=0.5pt, label=270: {$0$}] (0.1) at (0,0-2); 
\coordinate [circle, fill=black, inner sep=0.5pt, label=180: {$s_+$}] (s+.1) at (-0.2,1.4*2/4-2);
\coordinate [circle, fill=black, inner sep=0.5pt, label=180: {$s_-$}] (s-.1) at (-0.1,1.4*1/4-2);
\coordinate [circle, fill=black, inner sep=0.5pt, label=0: {$t$}] (t.1) at (0.35,1.4*3.5/4-2);

%hinge2
\coordinate [circle, fill=black, inner sep=0.5pt, label=270: {$0$}] (0.2) at (0+2,0); 
\coordinate [circle, fill=black, inner sep=0.5pt, label=180: {$s_+$}] (s+.1) at (-0.35+2,1.4*3.5/4);
\coordinate [circle, fill=black, inner sep=0.5pt, label=180: {$s_-$}] (s-.2) at (-0.05+2,1.4*0.5/4);
\coordinate [circle, fill=black, inner sep=0.5pt, label=0: {$t$}] (t.2) at (0.2+2,1.4*2/4);

%4pt2
\coordinate [circle, fill=black, inner sep=0.5pt, label=270: {$0$}] (0.2) at (0+2,0-2); 
\coordinate [circle, fill=black, inner sep=0.5pt, label=180: {$s_+$}] (s+.1) at (-0.35+2,1.4*3.5/4-2);
\coordinate [circle, fill=black, inner sep=0.5pt, label=180: {$s_-$}] (s-.2) at (-0.05+2,1.4*0.5/4-2);
\coordinate [circle, fill=black, inner sep=0.5pt, label=0: {$t$}] (t.2) at (0.2+2,1.4*2/4-2);

%hinge3
\coordinate [circle, fill=black, inner sep=0.5pt, label=270: {$0$}] (0.3) at (0+4,0); 
\coordinate [circle, fill=black, inner sep=0.5pt, label=180: {$s_+$}] (s+.3) at (-0.3+4,1.4*3/4);
\coordinate [circle, fill=black, inner sep=0.5pt, label=180: {$s_-$}] (s-.3) at (-0.2+4,1.4*2/4);
\coordinate [circle, fill=black, inner sep=0.5pt, label=0: {$t$}] (t.3) at (0.05+4,1.4*0.5/4); 

%4pt3
\coordinate [circle, fill=black, inner sep=0.5pt, label=270: {$0$}] (0.3) at (0+4,0-2); 
\coordinate [circle, fill=black, inner sep=0.5pt, label=180: {$s_+$}] (s+.3) at (-0.3+4,1.4*3/4-2);
\coordinate [circle, fill=black, inner sep=0.5pt, label=180: {$s_-$}] (s-.3) at (-0.2+4,1.4*2/4-2);
\coordinate [circle, fill=black, inner sep=0.5pt, label=0: {$t$}] (t.3) at (0.05+4,1.4*0.5/4-2);

%hinge4
\coordinate [circle, fill=black, inner sep=0.5pt, label=180: {$0$}] (0.4) at (0+6,0+0.5); 
\coordinate [circle, fill=black, inner sep=0.5pt, label=0: {$s_+$}] (s+.4) at (0.3+6,1*3/4+0.5);
\coordinate [circle, fill=black, inner sep=0.5pt, label=180: {$s_-$}] (s-.4) at (0.2+6,1.*2/4+0.5);
\coordinate [circle, fill=black, inner sep=0.5pt, label=0: {$t$}] (t.4) at (0.1+6,-1*1/4+0.5);

%4pt4
\coordinate [circle, fill=black, inner sep=0.5pt, label=180: {$0$}] (0.4) at (0+6,0+0.5-2); 
\coordinate [circle, fill=black, inner sep=0.5pt, label=0: {$s_+$}] (s+.4) at (0.3+6,1*3/4+0.5-2);
\coordinate [circle, fill=black, inner sep=0.5pt, label=180: {$s_-$}] (s-.4) at (0.2+6,1.*2/4+0.5-2);
\coordinate [circle, fill=black, inner sep=0.5pt, label=0: {$t$}] (t.4) at (0.1+6,-1*1/4+0.5-2); 

%labels of the four configs
\coordinate [label=270: {$(i)$}] (l1) at (0+0,2.2);
\coordinate [label=270: {$(ii)$}] (l2) at (0+2,2.2);
\coordinate [label=270: {$(iii)$}] (l3) at (0+4,2.2);
\coordinate [label=270: {$(iv)$}] (l4) at (0+6,2.2);
\end{scriptsize}
\end{tikzpicture}
\caption{The four possible subcases of endpoint-causal straight four-point configurations which arise from a hinge. 
}
\label{fig: 4pt-->mon}
\end{center}
\end{figure}
Let $s_+ > s_- >0$ and $t>0$ be such that $\alpha(s_+) \leq \beta(t)$ (the case of $\alpha(s_-) \leq \beta(t) \leq \alpha(s_+)$ follows analogously). 
These correspond to the cases (i) and (ii) in Figure \ref{fig: 4pt-->mon}, respectively. Note that in case (ii) we might deal with a causal four-point configuration if $\alpha(s_-) \leq \beta(t)$ are null related. 
Set $y:=\alpha(0)=\beta(0), x:= \alpha(s_-), z_1:=\alpha(s_+)$ and $z_2:=\beta(t)$. 
We need to show that $\theta(s_-,t) \leq \theta(s_+,t)$, i.e., $\tilde{\ma}_{y}(x,z_2) \geq \tilde{\ma}_{y}(z_1,z_2)$ (recall that $\theta$ is defined using signed angles). 
By construction, $(y, x, z_1, z_2)$ forms a left straight timelike and endpoint-causal four-point configuration. 
Let $(\hy, \hx, \hz_1, \hz_2)$ be a comparison configuration, then $\Delta(\hy,\hx,\hz_2)$ is a comparison triangle for $\Delta(y,x,z_2)$. 
Let $\Delta(\by,\bz_1,\bz_2)$ be a comparison triangle for (the possibly merely admissible causal triangle) $\Delta(y,z_1,z_2)$ and let $\bx$ be the comparison point for $x$ in $\Delta(\by,\bz_1,\bz_2)$. 
We have $\tilde{\ma}_{y}(x,z_2)=\ma_{\hy}(\hx, \hz_2)=\ma_{\hy}(\hz_1,\hz_2)$ and  $\tilde{\ma}_{y}(z_1,z_2)=\ma_{\by}(\bz_1,\bz_2)$, so the desired inequality reads $\ma_{\hy}(\hz_1,\hz_2) \geq \ma_{\by}(\bz_1,\bz_2)$. 
The two triangles $\Delta(\by,\bz_1,\bz_2)$ and $\Delta(\hy,\hz_1,\hz_2)$ have two sides of equal length, and by four-point comparison we know $\tau(\bz_1,\bz_2) = \tau(z_1,z_2) \geq \tau(\hz_1,\hz_2)$. 
Thus, $\tilde{\ma}_y(x,z_2)= \ma_{\hy}(\hz_1,\hz_2) \geq \ma_{\by}(\bz_1,\bz_2)=\tilde{\ma}_y(z_1,z_2)$ follows by law of cosines monotonocity, cf.\ \cite[Remark 2.5]{BS22}). 

For the remaining subcase of the future-directed case, let $s_-,s_+$ and $t$ be such that $\beta(t) \leq \alpha(s_-) \ll \alpha(s_+)$, see (iii) in Figure \ref{fig: 4pt-->mon} (note that also here one might deal with a causal four-point configuration if $\beta(t) \leq \alpha(s_-)$ are null related. 
Set $z_2 \coloneqq \alpha(0)=\beta(0), z_1 \coloneqq \beta(t), x \coloneqq \alpha(s_-)$ and $y \coloneqq \alpha(s_+)$ and consider the resulting left straight timelike and endpoint-causal past four-point configuration $(z_2, z_1, x, y)$. 

Construct a comparison configuration $(\hz_2, \hz_1, \hx, \hy)$ as well as a comparison triangle $\Delta(\bz_2, \bz_1, \by)$ for the triangle $\Delta(z_2, z_1, y)$. 
The triangles $\Delta(\bz_2, \bz_1, \by)$ and $\Delta(\hz_2, \hz_1, \hy)$ have two sides of equal length, and by four-point comparison we know $\tau(z_2,z_1)=\tau(\bz_2, \bz_1) \geq \tau(\hz_2, \hz_1)$. 
Thus, by law of cosines, we obtain $\ma_{\by}(\bz_2, \bz_1) \leq \ma_{\hy}(\hz_2, \hz_1)$. 
Let $\bx$ be a a comparison point for $x$ in $\Delta(\bz_2, \bz_1, \by)$ and consider the subtriangles $\Delta(\bz_1, \bx, \by)$ and $\Delta(\hz_1, \hx, \hy)$ of $\Delta(\bz_2, \bz_1, y)$ and $\Delta(\hz_2, \hz_1, \hy)$, respectively. 
They have two sides of equal length, and the angles at $\by$ (resp. $\hy$) agree with the ones in the original triangles, i.e., $\ma_{\by}(\bx, \bz_1) = \ma_{\by}(\bz_2, \bz_1) \leq \ma_{\hy}(\hz_2, \hz_1) = \ma_{\hy}(\hx, \hz_1)$. 
Thus, we get $\tau(\bx, \bz_1) \geq \tau(\hx, \hz_1) = \tau(x, z_1)$. 
Finally, we can relate a comparison triangle $\Delta(\bz_2', \bz_1', \bx')$ for $\Delta(z_2, z_1, x)$ to the subtriangle $\Delta(\bz_2, \bz_1, \bx)$ of $\Delta(\bz_2, \bz_1, \by)$. 
They have two sides of equal length, and from the above arguments we know $\tau(\bx, \bz_1) \leq \tau(x, z_1) = \tau(\bx', \bz_1')$. 
Hence the desired inequality $\ma_{\bz_2}(\bx, \bz_1) \geq \ma_{\bz_2'}(\by', \bz_1')$ follows. 

At last, consider the case of $\alpha$ being future-directed and $\beta$ being past-directed, see (iv) in Figure \ref{fig: 4pt-->mon}. 
In this case, for any choice of parameters, $\alpha(s_-), \alpha(s_+)$ and $\beta(t)$ (together with the origin) yield a straight timelike and endpoint-causal past four-point configuration. 
Labeling the points as in the case (iii) depicted in Figure \ref{fig: 4pt-->mon}, we observe that we are actually in the same situation, with the only difference being that the causal relation between $z_1$ and $z_2$ is reversed. 
Regardless, the arguments are completely analogous. 

The second part of the statement follows just as in Proposition \ref{pop: timelike vs causal 4pt condition}. 
\end{proof}

\begin{pop}[Causal vs.\ strict causal four-point condition]
\label{prop:4-point_causal_noncausal_equivalent}
Let $X$ be a \LpLS and let $U\subseteq X$ be open, $D_K$-geodesic, and regular. 
In the case of lower curvature bounds, assume in addition that $U$ is causally closed. 
Then $U$ is a $(\geq K)$- (resp.\ $(\leq K)$-)comparison neighbourhood in the sense of the causal four-point condition if and only if it is a $(\geq K)$- (resp.\ $(\leq K)$-)comparison neighbourhood in the sense of the strict causal four-point condition. 

In particular, if $X$ is strongly causal, regular and locally $D_K$-geodesic, then $X$ has curvature bounded below (resp.\ above) by $K$ in the sense of the causal four-point condition, if and only if it has the same bound in the sense of the strict causal four-point condition. 
\end{pop}
\begin{proof}
In both cases of implications, one implication is obvious from the definitions, so we only need to show that the causal four-point condition implies the strict causal four-point condition. 
For $(\leq K)$-comparison neighbourhoods, let first $(y, x, z_1, z_2)$ be a left straight causal and endpoint-causal four-point configuration. 
Then $x\ll z_1$ since $y \ll x$ and $U$ is regular. 
Let $\alpha:[0,1]\to X$ be the timelike distance realizer connecting $x$ to $z_1$, and set $z_1^t \coloneqq \alpha(t)$, then $z_1^t\ll z_2$ for all $t<1$. 
Consider the straight timelike four-point configuration $(y, x, z_1^t, z_2)$ and a comparison configuration $(\hat{y}, \hat{x}, \hat{z}_1^t, \hat{z}_2)$. 
Then we have $0<\tau(z_1^t,z_2)\leq\tau(\hat{z}_1^t,\hat{z}_2)$. 
Note that $(\hat{y}, \hat{x}, \hat{z}_1^t, \hat{z}_2)$ can be chosen so that it converges to a comparison configuration for $(y, x, z_1, z_2)$ as $t\nearrow 1$, i.e., $\hz_1^t \to \hz_1$. Thus, in the limit we get $\hat{z}_1 \leq \hat{z}_2$, as required. 

If $(y, x, z_1, z_2)$ is instead a right straight causal and endpoint-causal four-point configuration, let $\beta$ be the timelike distance realizer connecting $y$ to $z_1$ and $\gamma$ be the timelike distance realizer connecting $y$ to $x$. 
Then for all $t<1$ there is an $s<1$ such that $\gamma(t) \ll \beta(s)$, and we can make this choice $s_t$ continuously and such that $\lim_{t\nearrow 1} s_t = 1$. 
Set $x^t=\gamma(t)$ and $z_1^t=\beta(s_t)$. 
Consider the right straight timelike and endpoint-causal four-point configuration $(y, x^t, z_1^t, z_2)$, then we have $\tau(z_1^t,z_2) \geq \tau(z_1^t,z_1) > 0$. 
For a comparison configuration $(\hat{y}, \hat{x}^t, \hat{z}_1^t, \hat{z}_2)$, we have $0 < \tau(z_1^t,z_2) \leq \tau(\hat{z}_1^t,\hat{z}_2)$ by the causal four-point condition. 
In particular, $\hat{z}_1^t \leq \hat{z}_2$ for all $t \in [0,1)$. 
Note that $(\hat{y}, \hat{x}^t, \hat{z}_1^t, \hat{z}_2)$ can be chosen so that it converges to a comparison configuration for $(y, x, z_1, z_2)$ as $t\nearrow 1$, so in the limit we get $\hat{z}_1 \leq \hat{z}_2$, as required. 
\medskip

Now we have to look at whether $(\geq K)$-comparison neighbourhoods in the sense of the causal four-point condition are also such in the sense of the strict causal four-point condition. 
Let $(y, x, z_1, z_2)$ be a causal four-point configuration which is not necessarily endpoint-causal). 
For now, consider the case where $x\ll z_1$ is timelike. 
First, we look at the inequality between the $\tau$'s. 
Let $\alpha$ be the timelike distance realizer from $x$ to $z_1$, set $z_1^t \coloneqq \alpha(t)$ and consider the causal four-point configuration $(y, x, z_1^t, z_2)$. 
This yields a four-point comparison configuration $(\hat{y}, \hat{x}, \hat{z}_1^t, \hat{z_2})$ (note that this can be chosen in such a way that only $\hat{z}_1^t$ depends on $t$). 

By Proposition \ref{pop: four point condition implies monotonicity comparison}, we know that $U$ is also a comparison neighbourhood in the sense of monotonicity comparison. Thus, we have that $\tilde{\ma}_x(y,\alpha(t))=\ma_{\hat{x}}(\hat{y},\hat{z}_1^t)$ is increasing in $t$. 
By the triangle equality of angles in $\lm{K}$, we know $\ma_{\hat{x}}(\hat{z}_1^t,\hat{z}_2)=\ma_{\hat{x}}(\hat{y},\hat{z}_1^t)+\ma_{\hat{x}}(\hat{y},\hat{z}_2)$, hence also $\ma_{\hat{x}}(\hat{z}_1^t,\hat{z}_2)$ is increasing in $t$. 
Now we claim that $\tau(\hat{z}_1^t,\hat{z}_2)$ is strictly monotonically  decreasing in $t$ whenever $\hat{z}_1^t\leq \hat{z}_2$. 
Let $s<t$. 
Let $\hat{z}_1^{st}$ be the point on the side $[\hat{x}, \hat{z}_1^t]$ such that $\tau(\hx, \hz_1^{st}) = \tau(x,\alpha(s))$. Note that $\hat{z}_1^{st}$ is a comparison point on the side of a triangle, while $\hat{z}_1^s$ is a vertex of a comparison triangle. Then we observe: 
\begin{itemize}
\item $\tau(\hat{z}_1^{st},\hat{z}_2)>\tau(\hat{z}_1^{t},\hat{z}_2)$ by reverse triangle inequality ($\hat{z}_1^{st}\ll\hat{z}_1^{t}$),
\item $\tau(\hat{z}_1^{s},\hat{z}_2) > \tau(\hat{z}_1^{st},\hat{z}_2)$ by law of cosines monotonicity: These can each be completed to a triangle with $\hat{x}$. The other side lengths corresponding to each other agree and we know an inequality between the angles at $\hat{x}$. 
\end{itemize} 
In particular, we can look at the functions $f(t)=\tau(z_1^t,z_2)$ and $g(t)=\tau(\hat{z}_1^t,\hat{z}_2)$. 
We have just proven that $g$ is strictly monotonically decreasing whenever $\hz_1^t \leq \hz_2$. The reverse triangle inequality proves that $f$ is as well whenever $z_1^t \leq z_2$.  By Remark \ref{rem:endpoint-causality in strict 4pt}, both causal relations are certainly satisfied for small $t$. 
Ultimately, we have to show that $\tau(z_1,z_2) \geq \tau(\hz_1,\hz_2)$, i.e., $f(1) \geq g(1)$. We even show that $f(t) \geq g(t)$ for all $t \in (0,1]$. 
If $t$ is such that $f(t)>0$, we know $z_1^t \ll z_2$, so we can apply the causal four-point condition to the timelike and endpoint-causal four-point configuration $(y, x, z_1^t, z_2)$ and a corresponding comparison configuration to get that $f(t)\geq g(t)$. 
If $f(t)=0$, we have to show $g(t)=0$ as well. 
Let us now indirectly assume that there is $t_1$ such that $0 = f(t_1) < g(t_1) = \tau(\hz_1^t,\hz_2)$. 
By the reverse triangle inequality we then also have $\tau(\hx, \hz_2) \geq \tau(\hx, \hz_1^{t_1}) + \tau(\hz_1^{t_1}, \hz_2) > 0$. 
Again by Remark \ref{rem: relevant constellations of four-point configs}, we further have that $f(0)= \tau(x, z_2) = \tau(\hx, \hz_2) = g(0) > 0$, and thus for small enough $t$ also $f(t) > 0$ by continuity of $\tau$. In particular, there exists $t_0 \in (0,t_1)$ such that $f(t_0) > 0$. By the above argument, we gather $f(t_0) \geq g(t_0)$. 
As $f$ is continuous, there is a $t^* \leq t_1$ such that $f(t^*)=\min(\frac{g(t_1)}{2},f(t_0))$. 
Then we have $0<f(t^*)<g(t_1)\leq g(t^*)$ in contradiction to the causal four-point condition. 

As to the implication of the causal relations, recall that we are still in the case of $x \ll z_1$ and suppose towards a contradiction that $z_1 \not\leq z_2$ but $\hz_1 \leq \hat{z}_2$. Since $U$ is a causally closed neighbourhood, we infer that $\not\leq$ is open, i.e., $f(t)=0$ for $t$ close enough to $1$. 
However, since $\hz_1 \leq \hat{z}_2$, it follows that $\hz_1^t \ll \hz_2$ for all $t \in [0,1)$, which in turn gives $f(t) < g(t)$, a contradiction to the paragraph above. 

Finally, consider the case of $x\not\ll z_1$ being null related. 
We follow the same proof as above, but have to replace all the arguments leading to $g$ being strictly monotonically decreasing. 
In this case, $\alpha$ is a null curve from $x$ to $z_1$. In particular, we have $z_1^s\leq z_1^t$ for $s < t$ and hence $\tau(y,z_1^t) \geq \tau(y,z_1^s) + \tau(z_1^s,z_1^t)$ by the reverse triangle inequality. Moreover, the second term on the right hand side is zero (since $\alpha$ is null), which is why the reverse triangle inequality must be strict since otherwise regularity of $U$ would be violated. 
Thus, we conclude $\tau(y,z_1^t) > \tau(y,z_1^s)$. 
Arrange the four-point comparison configurations $(\hy, \hx, \hz_1^s, \hz_2)$ and $(\hy, \hx, \hz_1^t, \hz_2)$ such that they share the triangle $\Delta(\hy, \hx, \hz_2)$. 
Note that $\hx \leq \hz_1^s$ and $\hx \leq \hz_1^t$ are null related and point towards the left (by convention), hence $\hx, \hz_1^s$ and $\hz_1^t$ all lie on a null geodesic. 
In particular, $\hz_1^s$ and $\hz_1^t$ are causally related, and since $\tau(\hy,\hz_1^s) < \tau(\hy,\hz_1^t)$, we have $\hz_1^s \leq \hz_1^t$. 
By the reverse triangle inequality and the fact that $\hz_1^s, \hz_1^t$ and $\hz_2$ do not all lie on a single distance realizer, we have $\tau(\hz_1^s,\hz_2) > \tau(\hz_1^t,\hz_2)$ for all $s<t$ whenever $\hz_1^t\leq\hz_2$, i.e., $g$ is monotonically decreasing, and indeed strictly so on $\{ t \mid \hz_1^t \leq \hz_2 \} \subseteq [0,1]$. The rest of the proof works as in the case of $x \ll z_1$. 

The case of a past four-point configuration follows analogously.

The final claim of the proposition follows just as in Proposition \ref{pop: timelike vs causal 4pt condition}. 
\end{proof}

\subsection{Convexity and concavity of \texorpdfstring{$\tau$}{tau}}
Similar to the metric setting (cf., e.g., \cite{AKP19}), a characterization of curvature bounds via convexity or
concavity properties of modified distance functions relies crucially on the analytic properties of solutions
to the differential equation 
\begin{equation}\label{eq:K-convex}
f''-Kf=\lambda\,,
\end{equation}
its homogeneous variant
\begin{equation}\label{eq:K-convex-homogeneous}
f''-Kf=0\,,
\end{equation}
as well as the corresponding differential inequalities\footnote{This is a convexity (resp.\ concavity) condition on $f$, cf.\ \cite{AB03}. } $f''-Kf \ge \lambda$ (resp.\ $\le \lambda$). 
We therefore begin this
section by deriving some essentials of the solution theory for \eqref{eq:K-convex}. In the geometric applications we are interested in, the function $f$ will typically only be continuous. 
For such functions the standard solution concept is the distributional one (although for our purposes the most useful concept is `in the sense of Jensen', cf.\ Definition \ref{def:Jensen_def} below). 
Due to the hypoellipticity of any constant coefficient ordinary differential operator, this makes no difference in the case of equality in \eqref{eq:K-convex}, whose general solution is given explicitly by
\begin{equation}\label{eq:K-convex:explicit_solutions}
f(t) = 
\left\{
\begin{array}{ll}
    \alpha \cos(\sqrt{K}t) + \beta \sin(\sqrt{K}t) + \frac{\lambda}{K} & (K<0) \\
    \lambda \frac{t^2}{2} + \alpha t + \beta &  (K=0) \\
    \alpha \cosh(\sqrt{|K|}t) + \beta \sinh(\sqrt{|K|}t)  + \frac{\lambda}{K} & (K>0),
\end{array}
\right.
\end{equation}
where $\alpha, \beta \in \R$.  
In the inequality case, $f$ being a solution in the sense of distributions to $f''-Kf\ge \lambda$ (resp.\ $\le \lambda$), or a distributional \emph{subsolution} (resp.\ \emph{supersolution}), means that, for any smooth non-negative test function $\varphi$ with compact support in $I$ we have
\begin{equation}\label{eq:distributional_solution}
    \int_I f(t) \varphi''(t) - K f(t) \varphi(t) - \lambda \varphi(t) \,dt \ge 0 \quad (\text{resp.} \ \le 0).
\end{equation}
Such functions automatically are of higher regularity: 
\begin{pop}[Almost-convexity for distributional subsolutions]
\label{lem:almost_convexity_distributional}
Let $I\subseteq \mb{R}$ be an open interval, $f:I\to\mb{R}$ a continuous function which is a distributional subsolution
(resp.\ supersolution) to \eqref{eq:K-convex} and let $t_0\in I$. Then  there is a $c>0$ (resp.\ $c<0$) such that $f(t)+ct^2$ is a convex (resp.\ concave) function near $t_0$. In particular, $f$ is locally Lipschitz and possesses one-sided derivatives at every point of $I$. For these, we have $f'(t^-)\le f'(t^+)$ (resp.\ $f'(t^-)\ge f'(t^+)$) at each $t\in I$. 
\end{pop}
\begin{proof} Let $f$ be a distributional subsolution. 
By assumption, for any $c$ we have
\[
(f(t) + ct^2)'' \ge Kf(t) + \lambda+2c,
\]
and the right hand side can be made non-negative near $t_0$ for $c>0$ sufficiently big since $f$ is 
continuous, hence locally bounded. It follows that $f(t) + ct^2$ is a convex distribution, hence
a convex function near $t_0$ (cf.\ \cite[Theorem 4.1.6]{hoer1}). The remaining claims follow from well-known
properties of convex functions (cf.\ \cite[Corollary 1.1.6]{hoer_convexity}). The supersolution case follows from the subsolution one by considering $-f$.
\end{proof}
For the modified distance function we are going to study below, the following alternative solution concept will be relevant:
\begin{defi}\label{def:Jensen_def} A continuous function
$f:I\to \R$ ($I$ an interval) is called a solution to $f'' - K f \ge \lambda$ in the sense of Jensen\footnote{In \cite{AKP19} this property is called Jensen's inequality, which motivates our terminology here.} 
(ITSJ) if the following holds: If $t_1, t_2\in I$, $t_1<t_2$, $|t_1-t_2|<D_K$ and $g$ is the unique solution to \eqref{eq:K-convex} with $g(t_i) = f(t_i)$ ($i=1,2$), then $f(t)\le g(t)$ for all $t\in [t_1,t_2]$.
\end{defi}
Also here we speak of subsolutions\footnote{The defining inequality in Definition \ref{def:Jensen_def} is the reason we chose the names sub- and supersolution as we did. } ITSJ, and the supersolution case is defined analogously with the inequalities reversed. As in the distributional case, ITSJ solutions enjoy additional regularity properties:
\begin{pop}[Almost-convexity for Jensen subsolutions]\label{lem:almost_convexity_Jensen}
Let $I\subseteq \mb{R}$ be an open interval, $f:I\to\mb{R}$ a continuous function which is a subsolution (resp.\ supersolution) to \eqref{eq:K-convex} in the sense of Jensen and let $t_0\in I$. Then  there is a $c>0$ (resp.\ $c<0$) such that $f(t)+ct^2$ is a convex (resp.\ concave) function near $t_0$. In particular, $f$ is locally Lipschitz and possesses one-sided derivatives at every point of $I$. 
For these, we have $f'(t^-)\le f'(t^+)$ at each $t\in I$.
\end{pop}
\begin{proof} To show convexity of $f(t)+ct^2$ we need to establish that, for any $\lambda \in (0,1)$ and
$t_1<t_2$ near $t_0$ we have
\begin{equation}\label{eq:Jensen_almost_convexity}
\begin{split}
\lambda(f(t_1) + ct_1^2) + & (1-\lambda)(f(t_2) + ct_2^2) \\
& \ge f(\lambda t_1 + (1-\lambda)t_2) + c(\lambda t_1 + (1-\lambda)t_2)^2
\end{split}
\end{equation}
For $|t_1-t_2|$ small we can pick $g=g_{t_1,t_2}$ as in Definition \ref{def:Jensen_def} with
$g(t_i)=f(t_i)$ ($i=1,2$). Now picking $c=c(t_1,t_2)>0$ such that $g(t)+ct^2$ becomes convex near 
$t_0$ and inserting in \eqref{eq:Jensen_almost_convexity} we obtain
\begin{align*}
\lambda(f(t_1) & + ct_1^2) +  (1-\lambda)(f(t_2) + ct_2^2)  \\
&=\lambda(g(t_1)  + ct_1^2) + (1-\lambda)(g(t_2) + ct_2^2)\\
&\ge g(\lambda t_1 + (1-\lambda)t_2) + c(\lambda t_1 + (1-\lambda)t_2)^2\\
& \ge f(\lambda t_1 + (1-\lambda)t_2) + c(\lambda t_1 + (1-\lambda)t_2)^2.
\end{align*}
Hence the claim will follow once we are able to show that $c(t_1,t_2)$ can be chosen to remain uniformly bounded for $t_1 < t_2$
sufficiently near to $t_0$. Equivalently, we require a uniform lower bound on $g''$ in a fixed small neighbourhood of $t_0$. Now since $g$ is a solution to 
\eqref{eq:K-convex}, $g''$ is a solution to the corresponding homogeneous equation 
\eqref{eq:K-convex-homogeneous}. From the explicit formulae \eqref{eq:K-convex:explicit_solutions} the claim for $K=0$ follows immediately. So suppose first that
$K>0$. If $g''|_{[t_1,t_2]}$ 
has its minimum at $t_1$ or $t_2$ (e.g.\ if it is monotonous) 
then employing \eqref{eq:K-convex-homogeneous} for $g$ together
with the fact that $g(t_i)=f(t_i)$ ($i=1,2$) allows us to conclude local uniform boundedness of $g''$ from that of $f$.
Otherwise, $|g''|$ must attain a local maximum in $[t_1,t_2]$, say at $\hat t$. Then $|g''| \ge \frac{1}{2} |g''(\hat t)|$ in some ball $B_r(\hat t)$ around $\hat t$. 
Now \eqref{eq:K-convex-homogeneous} for $g''$ together with \eqref{eq:K-convex:explicit_solutions} shows that $g''(t) 
=A\cos(\sqrt{K}(t - \hat{t}))$ 
for suitable amplitude $A$, both depending on $t_1, t_2$. Since, however, the frequency $\sqrt{K}$ depends only on $K$, the same is true for the radius $r=r(K)$. Hence if we restrict to a ball of radius $r(K)$ around $t_0$, then $t_i$ will both
lie in $B_r(\hat t)$, so that (using \eqref{eq:K-convex} for $g$) we have
\[
|g''(\hat t)| < 2 |g''(t_1)| \le 2K |g(t_1)| + \lambda = 2K |f(t_1)| + \lambda,
\]
again allowing us to infer local boundedness of $g''$ around $t_0$ from that of $f$. In the case $K>0$, again by \eqref{eq:K-convex-homogeneous} and \eqref{eq:K-convex:explicit_solutions}, we can write
$g''$ in one of the forms $C\sinh(t+\theta)$, $C\cosh(t+\theta)$, 
$Ce^t$, or $Ce^{-t}$. Since none of these functions has an interior minimum on any finite interval, the claim follows as in the 
case of a boundary minimum above.
\end{proof}
We are going to need a comparison result on
the homogeneous version of \eqref{eq:K-convex}, which can be found in \cite[p. 23]{eschenburg}, cf.\ also \cite[Theorem 5.1.1]{kirchberger}:
\begin{lem}\label{lem:kirchberger} Let $\psi:[0,L]\to \R$ be a smooth solution to $\psi''-K\psi\ge 0$, $\psi(0)=0$,
$\psi(L)=0$ and assume that $L<D_K$. Then $\psi(t)\le 0$ for all $t\in [0,L]$.
\end{lem}
It is then immediate that the same conclusion holds if $\psi(0), \psi(L)$ are supposed to be $\le 0$ (cf.\ Corollary \ref{cor:strict_kirchberger} below for a strengthening of this result).  

The following result is a slight generalization of \cite[Theorem 3.14]{AKP19} (stated without proof there). Since we will repeatedly rely on arguments required for establishing it, we give a complete proof.
\begin{thm}\label{th:subsolutions_equivalence} Let $I = (a,b)\subseteq \R$ be an interval and $f: I \to \R$ a continuous function. The following are equivalent:
\begin{itemize}
\item[(i)] $f$ is a solution to $f'' - K f \ge \lambda$ in the distributional sense.
\item[(ii)] $f$ is locally Lipschitz and is a solution to $f'' - K f \ge \lambda$ in the support sense
(ITSS), i.e., for all $t_0\in I$ there is a solution $g:I\to\mb{R}$ of \eqref{eq:K-convex} with $g(t_0)=f(t_0)$ and $f\geq g$ on $[t_0-D_K,t_0+D_K]\cap I$.\footnote{Note that in \cite[Theorem 3.14 (b)]{AKP19} the inequality sign
has to be reversed.}
\item[(iii)]  $f$ is a solution to $f'' - K f \ge \lambda$ in the sense of Jensen.
\end{itemize}
The corresponding statement with all inequalities reversed in (i)--(iii) holds as well. 
\end{thm}
\begin{proof} Due to Propositions \ref{lem:almost_convexity_distributional} and \ref{lem:almost_convexity_Jensen} we may assume $f$ to be locally Lipschitz throughout.

(i)$\Rightarrow$(iii): Let $\varphi\in C_c^\infty((-1,1))$, $\varphi\ge 0$, $\int \varphi(t)\,dt=1$ and
set $\vphi_\eps(t) := \frac{1}{\eps} \vphi(\frac{t}{\eps})$. Finally, let $f_\eps := f * \vphi_\eps$. Then
$f_\eps$ is smooth on $I_\eps := (a+\eps,b-\eps)$ and satisfies $f_\eps'' - K f_\eps \ge \lambda$ on its domain.
Let $t_1<t_2 \in I$, $|t_1-t_2|<D_K$ and let $\eps>0$ be so small that $[t_1,t_2]\subseteq I_\eps$. Let $g_\eps$ be
the unique solution to \eqref{eq:K-convex} with $g_\eps(t_i)= f_\eps(t_i)$ ($i=1,2$). Then $\psi_\eps := f_\eps - g_\eps$ is smooth, $\psi_\eps'' - K\psi_\eps \ge 0$, and $\psi_\eps(t_i)=0$ for $i=1,2$. By Lemma \ref{lem:kirchberger}, then, $\psi_\eps\le 0$ on $[t_1,t_2]$. Letting $\eps\to 0$ implies that $f\le g$ on $[t_1,t_2]$ (cf.\ \cite[Theorem 11]{Klaasen}). 

(iii)$\Rightarrow$(ii): We first prove that, with $t_1, t_2$ and $g$ as in Definition \ref{def:Jensen_def}, i.e., as in (iii), we also have that $f(t)\ge g(t)$ 
for all $ t\in ((t_2-D_K,t_1] \cup [t_2,t_1+D_K]) \cap I$. We show this for any fixed $t_0\in (t_2-D_K,t_1]\cap I$, the other case being analogous. Denote by $h$ the solution to \eqref{eq:K-convex} with $h(t_0)=f(t_0)$ and $h(t_2)=f(t_2)$. By (iii),
$h\ge f$ on $[t_0,t_2]$ and we claim that $h\ge g$ on $[t_0,t_1]$: Supposing, to the contrary, the existence
of some $\bar t\in [t_0,t_1]$ with $h(\bar t)<g(\bar t)$, then since $h(t_1)\ge g(t_1)$ there would also
have to be some $\hat t \in [t_0,t_1]$ with $g(\hat t) = h(\hat t)$. But then since also $g(t_2)= f(t_2) = h(t_2)$, $g=h$ everywhere by unique
solvability of the boundary value problem for
\eqref{eq:K-convex}, giving a contradiction. Consequently, $g(t_0)\le h(t_0) = f(t_0)$, as claimed. 

Now for given $t\in I$ consider two sequences $s_n\nearrow t$, $t_n\searrow t$, and let  $g_n$ be the unique solution to \eqref{eq:K-convex} with $g_n(s_n) = f(s_n)$ and $g_n(t_n) = f(t_n)$. Since $f$ is locally Lipschitz,
it follows from \cite[Lemma 3]{FJ:62} that both $g_n$ and $g_n'$ remain uniformly bounded in some fixed neighbourhood
of $t$ for $n$ large. In particular, there is a subsequence of $(g_n(t),g_n'(t))$ that converges, so w.l.o.g.\ the whole sequence does. By continuous dependence on initial data, $g_n$ therefore converges to the solution $g$ with initial data $g(t)=\lim_n g_n(t)$ and $g'(t)=\lim_n g_n'(t)$.

For $t-D_K<s<t$ (the other case being analogous), we have for large enough $n$ that $t_n^+- D_K<s<t_n^-<t$. For such $n$ the above considerations show that $f(s)\geq g_n(s)$. Taking the limit as $n\to\infty$, we get $f(s)\geq g(s)$.

(ii)$\Rightarrow$(i): Assume first that $f$ is smooth and let $t_0\in I$ and $g$ as in (ii). Then
\begin{align}
f''(t_0) &= \lim_{h\to 0} \frac{f(t_0+h) - 2f(t_0) + f(t_0-h)}{h^2}\\ 
&\ge  \lim_{h\to 0} \frac{g(t_0+h) - 2g(t_0) + g(t_0-h)}{h^2}\\ 
&= g''(t_0) = \lambda + Kg(t_0) = \lambda + Kf(t_0)
\end{align}
gives the claim in this case. Next, suppose $f$ is only locally Lipschitz and let $\vphi_\eps$ be as above. 
For $s\in I$ let $g_s$ be the solution from (ii) with $t_0=s$. If $f$ is differentiable at $s$ then (ii) implies
that $g_s'(s) = f'(s)$. For such values of $s$, therefore, $g_s$ is uniquely determined and can indeed be calculated
from $f(s)$ and $f'(s)$ according to \eqref{eq:K-convex:explicit_solutions}. Since $f'\in L^\infty_{loc}(I)$ it follows
that $(s,t)\mapsto g_s(t) \in L^\infty(I\times I)$. Now fix $t_0\in I$ and set (for $t$ near $t_0$ and $\eps$ small)
\[
h_\eps(t) := \int g_{t_0-s}(t-s) \vphi_\eps(s)\, ds = \int g_{t_0-s}(s) \vphi_\eps(t-s)\, ds.
\]
Then $h_\eps$ is a smooth solution to $h_\eps'' - Kh_\eps = \lambda$. Moreover,
\[
f_\eps(t_0) = \int f(t_0-s)\vphi_\eps(s) \,ds = \int g_{t_0-s}(t_0-s) \vphi_\eps(s)\, ds = h_\eps(t_0)
\]
and $f_\eps(t) \ge h_\eps(t)$ near $t_0$ (for $\eps$ small). From the smooth case it then follows that
$f_\eps'' - K f_\eps \ge \lambda$ near $t_0$, so $\eps\to 0$ gives the claim.
\end{proof}
We record the following consequence of the proof of (iii)$\Rightarrow$(ii) in Theorem \ref{th:subsolutions_equivalence}:
\begin{cor}[Outer Jensen]\label{cor:outer_jensen}
Let $I\subseteq \mb{R}$ be an open interval and $f:I\to\mb{R}$ locally Lipschitz. Then $f$ is a subsolution in the sense of 
Jensen if and only if it satisfies the following condition: for all $t_1<t_2\in I$ the (unique) solution $g:I\to\mb{R}$ of \eqref{eq:K-convex} with $g(t_i)=f(t_i)$ satisfies $f\geq g$ on $(I\setminus (t_1,t_2))\cap(t_2-D_K,t_1+D_K)$.
\end{cor}

The following solutions of \eqref{eq:K-convex} will be of particular interest to us:
\begin{defi}
The \emph{modified distance} function $\md^K:[0,D_K) \to \R$ in the model space $\lm{K}$ is the solution of the initial
value problem
\begin{equation}
\begin{cases}
(\md^K)''-K\md^K=1\\
\md^K(0)=0\\
(\md^K)'(0)=0
\end{cases}
\end{equation}
The \emph{modified sine} function $\sn^K=(\md^K)'$ and the \emph{modified cosine} function
$\cn^K=(\sn^K)'$ are the solutions of the following initial value problems for the corresponding homogeneous equation:
\begin{equation}
\begin{cases}
(\sn^K)''-K\sn^K=0\\
\sn^K(0)=0\\
(\sn^K)'(0)=1
\end{cases}
\end{equation}
and
\begin{equation}
\begin{cases}
(\cn^K)''-K\cn^K=0\\
\cn^K(0)=1\\
(\cn^K)'(0)=0
\end{cases}
\end{equation}
Explicitly:\\

\begin{center}
\begin{tabular}{c|c|c|c}
&$\md^K(t)$&$\sn^K(t)$&$\cn^K(t)$  \\ \hline
$K=0$&$\frac{t^2}{2}$&$t$&$1$\\ \hline
$K=1$&$\cosh(t)-1$&$\sinh(t)$&$\cosh(t)$\\ \hline
$K=-1$&$1-\cos(t)$&$\sin(t)$&$\cos(t)$
\end{tabular}
\end{center}
\end{defi}

The role of the modified distance is the following, compare \cite[Chapter 1, 1.1(a)]{AKP19}:
\begin{lem}[$\md^K$ and geodesics in model spaces]\label{lem:modified_distance_in_model_spaces}
Let $p\in\lm{K}$ and $\gamma:\mb{R}\to\lm{K}$ be a $\tau$-unit speed geodesic. 
Then the partial function 
\begin{equation} 
\label{eq:mdTauTwoPart}
f(t)=\begin{cases}
    \md^K(\tau(p,\gamma(t)))&p\leq\gamma(t)\\
    \md^K(\tau(\gamma(t),p))&\gamma(t)\leq p
\end{cases}
\end{equation}
satisfies:
\begin{equation}
\label{eq:convexity_with_lambda1}
f''-K f=1 \,,
\end{equation}
more precisely, $f=\md^K+C\sn^K+D\cn^K
$ where this is positive, and $f$ is not defined where the right hand side is negative, with $C=f'(0)$ and $D=f(0)$. In particular, it is extensible to a solution on $\mb{R}$.
\end{lem}
\begin{proof}
By rescaling it suffices to consider the cases $K=-1,0,+1$. For $K=0$, after applying a suitable Lorentz transformation we can assume  that $p=(t_p,x_p)$ and $\gamma(t)=(t,0)$. It then follows that $f(t)=\frac{t^2}{2}-t_pt + \frac{t_p^2-x_p^2}{2}  =\md^0(t)+f'(0)\sn^0(t) + f(0)\cn^0(t)$.

For the case $K=1$, again by applying a suitable\footnote{Here, we view de Sitter space as the set $\{(t,x,y) \mid -t^2+x^2+y^2=+1\}$ embedded in $\R^3_1$.} Lorentz transformation we may
assume that $p=(t_p,x_p,y_p)$ and $\gamma(t)=(\sinh(t),0,\cosh(t))$. 
Now for any causally related points $v,w$ in de Sitter space $\lm{1}$, the time separation function
is given explicitly by $\tau(v,w) = \arcosh \langle v, w\rangle$ (cf., e.g., \cite[(2.7)]{CRHK}).
Consequently, for $p\le \gamma(t)$, $\tau(p,\gamma(t)) = \arcosh(-t_p\sinh{t} + y_p \cosh{t})$, and so
\begin{align*}
    f(t)=\md^1(\tau(p,\gamma(t))) &= -t_p\sinh{t} + y_p \cosh{t} - 1\\
        &= \md^1(t) + f'(0)\sn^1(t) + f(0)\cn^1(t).
\end{align*}
Finally, for the anti-de Sitter\footnote{Here, we view anti-de Sitter space as the set $\{(s,t,x) \mid -s^2-t^2+x^2=-1\}$ embedded in $\R^3_2$.} case $K=-1$, we can w.l.o.g.\ assume $\gamma(t)=(\cos(t),\sin(t),0)$, and we set $p=(s_p,t_p,x_p)$. Here, for causally related points $v,w$ we have $\tau(v,w) = \arccos(- \langle v, w\rangle)$ (using \cite[(2.7)]{CRHK}, together
with \cite[Lem.\ 4.24]{ON83}). Thus for $p\le \gamma(t)$ we have
\begin{align*}
    f(t) &= \md^{-1}(\tau(p,\gamma(t))) = 1 + \langle p,\gamma(t) \rangle 
        = 1-s_p\cos{t} - t_p \sin{t}\\
        &= \md^{-1}(t) + f'(0) \sn^{-1}(t) + f(0)\cn^{-1}(t).
\end{align*}
\end{proof}

The definition of $f$ in \eqref{eq:mdTauTwoPart} requires a distinction in cases depending on the causal relation of $p$ and $\gamma(t)$. 
This is also the main difference with respect to the metric machinery: if there is no causal relation between $p$ and $\gamma(t)$, then $f$ cannot give any information. 
This is why the next few results build up  the theory in order to extend some known results from metric geometry to domains which consist of two intervals.

\begin{lem}
\label{lem:behaviour_of_solutions} 
Let $f_i:(a,b)\to\mb{R}$, $i=1,2$, be two solutions of \eqref{eq:K-convex} and let $t_1, t_2\in I$ with $t_1<t_2$. Then if $f_1(t_1)=f_2(t_1)$ and $f_1(t_2)<f_2(t_2)$, we have $f_1<f_2$ on $(t_1,\min(b,t_1+D_K))$ and $f_1>f_2$ on $(\max(a,t_1-D_K),t_1)$.
%\todo{We could simplify these intervals by instead assuming $b-t_1,t_1-a<D_K$.}
In particular, if $f_1(t_1)>f_2(t_1)$ and $f_1(t_2)<f_2(t_2)$ we have $f_1<f_2$ on $(t_2,\min(b,t_1+D_K))$ and $f_1>f_2$ on $(\max(a,t_2-D_K),t_1)$. 
\end{lem}
\begin{proof}
Assume w.l.o.g.\ that $t_1=0$ and let $f:=f_2-f_1$. 
Then $f$ is a solution of the homogeneous problem \eqref{eq:K-convex-homogeneous}, hence it is of the form $f=a\cn^K+b\sn^K$. 
As $f(0)=0$, we have $f=b\sn^K$ and $b=\frac{f(t_2)}{\sn^K(t_2)}>0$. By the explicit formula, $\sn^K(t)>0$ for $t\in(0,D_K)$ and $\sn^K(t)<0$ for $t\in(-D_K,0)$. 
For $f_1$ and $f_2$ and general $t_1$ this means that $f_1<f_2$ on $(t_1,\min(b,t_1+D_K))$ and $f_1>f_2$ on $(\max(a,t_1-D_K),t_1)$.

For the ``in particular'' statement, note that $f$ has a zero between $t_1$ and $t_2$, so we can apply the main part of the lemma. 
\end{proof}
As a consequence, we obtain the following strengthening of Lemma \ref{lem:kirchberger}: 
\begin{cor}\label{cor:strict_kirchberger} 
Let the continuous function $\psi: [0,L] \to \R$ be a solution to $\psi''-K\psi\ge 0$ in the sense of Jensen
%Under the assumptions of Lemma \ref{lem:kirchberger}, 
and assume that $\psi(0)\le0$ and $\psi(L)\le0$, with at least one of these inequalities strict. Then $\psi(t)<0$ for all $t\in (0,L)$. 
\end{cor}
\begin{proof} Let $\vphi_1$ be the unique solution to \eqref{eq:K-convex-homogeneous} with $\vphi_1(0)=\psi(0)$
and $\vphi_1(L)=0$, and $\vphi_2$ the one with $\vphi_2(0)=0$ and
$\vphi_2(L) = \psi(L)$. Then $\vphi := \vphi_1+\vphi_2$ is the unique solution to \eqref{eq:K-convex-homogeneous} with 
the same boundary conditions as $\psi$, so $\psi\le \vphi$ on $[0,L]$.
% Applying Lemma \ref{lem:kirchberger} to $\tilde\psi:= \psi+\vphi$ implies
% that $\psi\le -\vphi$ on $[0,L]$. 
This proves the claim since at least one of $\vphi_1$ and $\vphi_2$ is strictly negative on $(0,L)$ by 
Lemma \ref{lem:behaviour_of_solutions}.
\end{proof}

In the following results, we will require a Jensen-type solution concept that is applicable to domains 
more general than intervals: 

\begin{defi} Let $U\subseteq \R$ be any subset. A continuous function $f:U\to \R$
is said to satisfy the \emph{Jensen subsolution (resp.\ supersolution) inequality} for the parameters $(t_1,t_2,t_3)$ with $t_1<t_2<t_3$
and $|t_1-t_3|<D_K$ if $f(t_2) \le g(t_2)$ (resp.\ $f(t_2) \geq g(t_2)$), where $g$ is the unique solution to \eqref{eq:K-convex} with $g(t_i)=f(t_i)$ for $i=1,3$.
\end{defi}

\begin{rem}[Reformulation of subsolutions (resp.\ supersolutions)]
\label{rem: reformulation subsolution supersolution jensen}
Note that $f$ is a subsolution (resp.\ supersolution) in the sense of Jensen if and only if it satisfies the 
Jensen subsolution (resp.\ supersolution) inequality for all parameters $(t_1,t_2,t_3)$ with $t_1<t_2<t_3$
and $|t_1-t_3|<D_K$.    
\end{rem}

\begin{pop}[Splitting Jensen]\label{prop:splitting_jensen}
Let $I\subseteq \mb{R}$ be an open interval, $f:I\to\mb{R}$ a continuous function. Let $t_1<t_2<t_3<t_4$ and $|t_1-t_4|<D_K$. If $f$ satisfies the Jensen subsolution (supersolution) inequality for the parameters $(t_1,t_2,t_3)$ and for $(t_2,t_3,t_4)$, it also satisfies it for $(t_1,t_2,t_4)$ and $(t_1,t_3,t_4)$. 
\end{pop}
\begin{proof} Again it will suffice to prove the subsolution case. Let 
$g_{13}$, $g_{24}$ and $g_{14}$ be the solutions to \eqref{eq:K-convex} with 
$g_{13}(t_1)=f(t_1)$, $g_{13}(t_3)=f(t_3)$, $g_{24}(t_2)=f(t_2)$ and $g_{24}(t_4)=f(t_4)$, $g_{14}(t_1)=f(t_1)$ and $g_{14}(t_4)=f(t_4)$. Then our assumption is that $f(t_2)\le g_{13}(t_2)$, as 
well as $f(t_3)\leq g_{24}(t_3)$. We then have to show,  w.l.o.g., that $f(t_2)\leq g_{14}(t_2)$. 

As $g_{13}(t_2)\geq g_{24}(t_2)$ and $g_{13}(t_3)\leq g_{24}(t_3)$, we can use Lemma \ref{lem:behaviour_of_solutions} to conclude that $g_{24}(t_1)\leq f(t_1)$. Now set $\psi:=g_{24}-g_{14}$. Then $\psi(t_1)\le 0$ and $\psi(t_4)=0$,
so Lemma \ref{lem:kirchberger} implies that $\psi\le 0$ on $[t_1,t_4]$. In particular, $f(t_2) = g_{24}(t_2)\le g_{14}(t_2)$.

\end{proof}
Note that the above result could in fact be formulated entirely in terms of properties of $g_{13}$, $g_{24}$ and $g_{14}$, without recourse to the function $f$. However, for the applications we have in mind the formulation we chose is more appropriate. 

As in the case of the model spaces (see Lemma \ref{lem:modified_distance_in_model_spaces}) we 
will be interested in when a function $f$ defined on two disjoint intervals is a subsolution of \eqref{eq:K-convex} when restricted to the set where it is non-negative. For our intended purposes (cf.\ Definition \ref{def-cb-taucvx} below), these intervals will be closed. In this case we note
the following immediate but helpful consequence of Definition \ref{def:Jensen_def}:
\begin{rem}\label{rem:subsolutions_equivalence_other_notions}
If $f$ is a continuous function on an interval $[a,b]$, then $f$ is a subsolution (resp.\ supersolution) to \eqref{eq:K-convex} ITSJ on $[a,b]$ if and only if it is one on $(a,b)$.
It follows from this, together with Theorem \ref{th:subsolutions_equivalence}, that in the
following result the Jensen solution concept on intervals can also be expressed in terms of
the equivalent notions given there.
\end{rem}

\begin{pop}
\label{pop: split domain equivalences}
Let $I=[a,b]\cup[c,d]$, $a<b<c<d$, and $f:I\to\mb{R}$ continuous with $f(b)=f(c)=0$. Then the following are equivalent:
\begin{enumerate}
    \item $f$ is a subsolution ITSJ of \eqref{eq:K-convex} in the following sense: 
    For all $t_1,t_2,t_3\in [a,b]\cup[c,d]$ with $t_1<t_2<t_3$ and $|t_1-t_3|<D_K$, $f$ satisfies the
    Jensen subsolution inequality.
    \item $f$ is extensible as a subsolution of \eqref{eq:K-convex} ITSJ on $[a,d]$.
    \item $f$ is a subsolution of \eqref{eq:K-convex} ITSJ on both parts of its domain, and for the solution $g$ of \eqref{eq:K-convex} with $g(b)=g(c)=0$ we have $f'(b^-)\leq g'(b)$ and $f'(c^+)\geq g'(c)$.
\end{enumerate}
An analogous equivalence holds for supersolutions. 
\end{pop}
\begin{proof}
(ii)$\Rightarrow$(i) is clear.

(i)$\Rightarrow$(iii): The proof of Corollary \ref{cor:outer_jensen} still works in the present setup and
shows that for $g$ as in (iii) we have $g\le f$ on $[\max(a,c-D_K),b]$ (as well as on $[c,\min(d,a+D_K)]$).
Therefore, 
\begin{equation*}
f'(b^-)=\lim_{t\nearrow b}\frac{f(b)-f(t)}{b-t}\leq\lim_{t\nearrow b}\frac{g(b)-g(t)}{b-t}=g'(b).
\end{equation*}
The inequality $f'(c_+) \geq g'(c)$ can be established in a similar
way.

(iii)$\Rightarrow$(ii): Taking $g$ as in (iii), we extend $f$ to $[a,d]$ by setting $f(t):=g(t)$ for $t\in(b,c)$. We now verify the Jensen condition for this extended function. 
By Proposition \ref{prop:splitting_jensen}, we only need to check three cases: $|t_1-t_3|<D_K$ and either $t_1<t_2<t_3$ are all contained in either $[a,b]$ or $[b,c]$ or $[c,d]$ (where it is automatically satisfied), or $(t_1,b,t_3)$ with $t_1\in (a,b)$ and $t_3\in (b,c)$, or, finally, $(t_1,c,t_3)$ with $t_1\in (b,c)$ and $t_3\in (c,d)$. 
By symmetry, we only need to treat $(t_1,b,t_3)$.
Thus let $h$ be the unique solution to \eqref{eq:K-convex} with $h(t_1)=f(t_1)$ and $h(t_3)=f(t_3)$. The claim then is
that $h(b)\ge f(b)$. 

If we can show that $f\ge g$ on $[t_1,b]$ we will be done: Indeed then $h(t_1)=f(t_1)\ge g(t_1)$ and $h(t_3)=g(t_3)$,
so Lemma \ref{lem:kirchberger} implies $h(b)\ge g(b) = f(b)$. So we are left with proving that $k:=f-g$, which
is a subsolution ITSJ of the homogeneous equation \eqref{eq:K-convex-homogeneous}, is non-negative on $[t_1,b]$.

We now construct a supporting solution for $k$ at $b$: Let $s_n \in (t_1,b)$, $s_n\nearrow b$ and denote 
by $l_n$ the unique solution
to \eqref{eq:K-convex-homogeneous} with $l_n(s_n)=k(s_n)$ and $l_n(b)=k(b)$. As in the proof of Theorem \ref{th:subsolutions_equivalence}, (iii)$\Rightarrow$(ii) it then follows that (up to picking a subsequence)
$l_n$ converges (in $C^2$) to the solution $l$ of the initial value problem to \eqref{eq:K-convex-homogeneous}
with 
\begin{align*}
   l(b) &=\lim_n l_n(b) =  k(b)\\
   l'(b) &=\lim_n l_n'(b).
\end{align*}
By the mean value theorem, $\frac{l_n(b) - l_n(s_n)}{b-s_n} = l_n'(\tilde s_n) \to l'(b)$ (where $\tilde s_n \in (s_n,b)$), which by construction implies that $l'(b) = k'(b^-)$. 
As in the proof of Theorem \ref{th:subsolutions_equivalence}, (iii)$\Rightarrow$(ii), it follows 
that $l$ is a supporting function for $k$ at $b$, i.e., $l\le k$ on $[t_1,b]$. Furthermore,
we have $l(b) = k(b) =0$ and $l'(b) = k'(b^-)\le 0$ by (iii), so
 \eqref{eq:K-convex:explicit_solutions} implies that $l\ge 0$ on
$[t_1,b]$. Thus, finally, $k\ge l \ge 0$, so $f \ge g$ on $[t_1,b]$, as claimed.
\end{proof}

We now have the necessary tools at hand to introduce a characterization of curvature bounds via a convexity/concavity property of $\tau$, see \cite[Theorems 8.23 \& 9.25]{AKP19}. 

\begin{defi}[Curvature bounds by convexity/concavity of $\tau$]
\label{def-cb-taucvx}
Let $X$ be a regular \LpLSn. An open subset $U$ is called a $(\geq K)$- (resp.\ $(\leq K)$-)comparison neighbourhood in the sense of the \emph{$\tau$-convexity (resp.\ $\tau$-concavity) condition} if:
\begin{enumerate}
\item $\tau$ is continuous on $(U\times U) \cap \tau^{-1}([0,D_K))$, and this set is open.
\item $U$ is $D_K$-geodesic. 
\item \label{def-cb-taucvx.main} Let $p\in U$ and let $\gamma:[a,d]\to U$ be a
timelike
$\tau$-arclength parametrized distance realizer\footnote{Any timelike distance realizer can be parametrized by $\tau$-arclength, but this parametrization need not be Lipschitz, cf.\ \cite[Corollary 3.35]{KS18}.} with $\tau(p,\gamma(d))<D_K$, $\tau(\gamma(a),p)<D_K$ and $\tau(\gamma(a),\gamma(d))=d-a<D_K$. We %set
define the partial function on $[a,d]$
\begin{equation}\label{eq:mdTauTwoPartX}
    f(t)=
    \begin{cases}
    \md^K(\tau(p,\gamma(t))), \text{ if } p\leq\gamma(t)\\
    \md^K(\tau(\gamma(t),p)), \text{ if } \gamma(t)\leq p \, ,
    \end{cases}
\end{equation} 
(compare this with \eqref{eq:mdTauTwoPart}, but note that here $\tau$ denotes the time separation in $X$). If $f$ is not defined on a closed subset, extend it by setting it equal to $0$ on the boundary of its domain. We require
\begin{equation}
f''-Kf\geq 1  \quad \text{ (resp.\ } f''-Kf\leq 1 \text{)} \, ,
\end{equation}
i.e., $f$ is a subsolution (resp.\ supersolution) in the sense of any of the equivalent formulations established in Proposition \ref{pop: split domain equivalences} and Remark \ref{rem:subsolutions_equivalence_other_notions}. 
\end{enumerate}
\end{defi}

\begin{rem}[Domain of $f$]
Define $b=\sup\{t\in[a,d]:\gamma(t)\leq p\}$, $c=\inf\{t\in[a,d]:p\leq\gamma(t)\}$ (if the respective set is nonemtpy).
Then the function $f$ defined in \eqref{eq:mdTauTwoPartX} has the domain:
$[a,b]\cup [c,d]$ if both $b,c$ are defined, $[a,b]$ if $c$ is not defined, $[c,d]$ if $b$ is not defined, and $\emptyset$ if neither are defined. 
Without the extension of $f$ and if $U$ is not causally closed, the points $b,c$ may be missing from these sets.
\end{rem}

\begin{pop}[Triangle comparison and $\tau$- convexity (resp.\ concavity) condition are equivalent]
\label{pop: triangle comparison and tau convex are equivalent}
Let $U$ be an open subset in a regular \LpLS $X$. 
Then $U$ is a $(\geq K)$- (resp.\ $(\leq K)$-)comparison neighbourhood in the sense of one-sided timelike triangle comparison if and only if it is a $(\geq K)$ (resp.\ $(\leq K)$) -comparison neighbourhood in the sense of the $\tau$-convexity (resp.\ $\tau$-concavity) condition. 
\end{pop}

\begin{proof} Since $X$ is assumed to be regular, the first two conditions in Definitions \ref{def-cb-tr} and \ref{def-cb-taucvx} agree. 
It is left to check the third condition. We will do this for lower curvature bounds and mention where the case of upper curvature bounds is not analogous. 

Below, when considering the convexity condition, we take a point $p \in U$ and a timelike distance realizer $\gamma$ in $\tau$-unit speed parametrization. 
Define the partial function $f$ as required, without the extension (cf.\ \eqref{eq:mdTauTwoPartX}). 
Note that the extension of $f$ is still continuous and, by a limit argument, the extended $f$ is a subsolution (resp.\ supersolution) ITSJ if and only if $f$ before the extension was a subsolution (resp.\ supersolution) ITSJ. 
We then want to check that $f$ is a subsolution of 
\eqref{eq:K-convex} ITSJ with $\lambda=1$ by showing that it satisfies the Jensen subsolution inequality for any parameters $t_1<t_2<t_3$ with $|t_1-t_3|<D_K$, cf.\ Remark \ref{rem: reformulation subsolution supersolution jensen}. 
Take $t_1<t_2<t_3$ in the domain of $\gamma$ with $|t_1-t_3|<D_K$. 
We set $x=\gamma(t_1)$, $q=\gamma(t_2)$ and $y=\gamma(t_3)$, and assume that $x$ and $y$ are causally related to $p$. 
%We will consider all cases resulting from this.

On the other hand, when considering one-sided triangle comparison, we will let $\Delta(x,p,y)$ form a timelike triangle (in any possible permutation), and let $q$ be a point on the side $[x,y]$ realized by $\gamma$ in $\tau$-unit speed. 
Let $x=\gamma(t_1)$, $q = \gamma(t_2)$ and $y=\gamma(t_3)$. 
Note that choosing $q$ in the interior of $[x,y]$ is not a real restriction since otherwise we would have trivial equality of the $\tau$-lengths in $X$ and $\lm{K}$.  

Now we are ready to establish both directions simultaneously. 

\emph{Case 1:} 
First assume that the parameters are suitable for both timelike triangle comparison as well the Jensen inequality. 
More precisely, assume that $x,y$ are timelike related to $p$ and $p\leq q$ (in which direction the points are related is not important). 
In particular, $t_1,t_2,t_3$ all lie inside the domain of $f$. 
Then we can consider both the one-sided triangle comparison for $q$ in $\Delta(p,x,y)$ and the Jensen inequality for the parameters $t_1<t_2<t_3$, and we claim that they are equivalent.

For both triangle comparison and Jensen inequality, we take a comparison situation $\Delta(\bar{p},\bar{x},\bar{y})$ and a comparison point $\bar{q}$. In the present case 1, we additionally assume $\bar{p}\leq\bar{q}$. 
Define $\bar{\gamma}$ as the side $[\bar{x}, \bar{y}]$ in $\tau$-unit speed parametrization such that $\bar{\gamma}(t_1)=\bar{x}$, then $\bar{\gamma}(t_2)=\bar{q}$ and $\bar{\gamma}(t_3)=\bar{y}$. 

We set
\[
\bar{f}(t):=\begin{cases}
    \md^K(\tau(\bar p,\bar{\gamma}(t))), & \text{ if } \bar{p} \leq \bar{\gamma}(t) \, , \\
    \md^K(\tau(\bar{\gamma}(t), \bar p)), & \text{ if } \bar{\gamma}(t) \leq \bar{p} \, .
    \end{cases}
\]
By Lemma \ref{lem:modified_distance_in_model_spaces},  $\bar{f}$ is a solution of \eqref{eq:convexity_with_lambda1} and is given by $\md^K+B\cn^K+C\sn^K$ where both are defined, which is precisely where the latter is non-negative. 
Extend $\bar{f}$ to $g$, given by that formula. We have that $f(t_1)=\bar{f}(t_1)$ and $f(t_3)=\bar{f}(t_3)$, so $g$ is the solution of \eqref{eq:convexity_with_lambda1}  with $f(t_1)=g(t_1)$ and $f(t_3)=g(t_3)$.

Now one-sided triangle comparison from below precisely says that $\tau(p,q)\leq\tau(\bar{p},\bar{q})$. 
The case of $\tau(q,p)=0\leq\tau(\bar{q},\bar{p})=0$ is automatic: indeed, by push-up and chronology, $p \leq q$ implies $q \not\ll p$ (and similarly in the model space). 
As $\md^K$ is strictly monotonically increasing on $[0,D_K]$, this is equivalent to 
\begin{equation}
f(t_2) = \md^K(\tau(p,q)) \leq \md^K(\tau(\bar{p},\bar{q})) = \bar{f}(t_2) \, ,
\end{equation}
which is the Jensen subsolution inequality for $t_1<t_2<t_3$. 
The curvature bounded above case follows analogously, concluding Case 1.

For all of the remaining cases, note that the logic is more subtle: we aim to prove that all Jensen inequalities imply the desired curvature bound inequality, and that all curvature bound inequalities imply the desired Jensen inequality.

Before we continue, note that we can restrict to admissible causal triangles in the Jensen inequality. 
Indeed, it cannot be that both $x,y$ are not timelike related to $p$: as $t_1<t_3$, this could only be the case if we have $x\leq p\leq y$ all null related, but then $p,q$ are causally unrelated, contrary to our assumption in the present case. 
In all of the following cases, we will therefore assume that $x, y$ and $p$ form an admissible causal triangle. 

\emph{Case 2:} We can extend Case 1 to admissible causal triangles. Thus, let $x,y,q$ be causally related to $p$ and not both $x$ and $y$ timelike related to $p$ as well as $\bp$ and $\bq$ causally related (in the same direction as $p$ and $q$), then timelike triangle comparison does not make sense immediately. 
If only one of $x,y$ is not timelike related, by varying $p$ and $q$, one can reduce this to Case 1 by a limiting procedure as in the second paragraph in Proposition \ref{prop:4-point_causal_noncausal_equivalent}. 

We have now dealt with all cases where, up to symmetry, $p \leq q$ and $\bp \leq \bq$. 

Furthermore, we have to consider curvature bounds above and below separately. 

\emph{Case 3}: Let now $p\leq q$ and $\bar{p},\bar{q}$ be causally unrelated. 
Then we automatically have $0\leq f(t_2)$ and $0>g(t_2)$. 
In the case of curvature bounds below, we use the main argument in the equivalence of strict causal curvature bounds (Theorem  \ref{thm:equivCausTl} (ii)). 
As the implication $p\leq q\Rightarrow \bar{p}\leq\bar{q}$ does not hold, this contradicts strict causal triangle comparison, which we know to be equivalent to one-sided timelike triangle comparison by Proposition \ref{rem: one-sided triangle comparison and triangle comparison are equivalent} and Theorem \ref{thm:equivCausTl}, thus also some curvature bound inequality for $\tau$ fails. 
For the Jensen subsolution inequality, $f(t_2)\leq g(t_2)$ is violated, making these match. 

Triangle comparison above is automatically satisfied. For the Jensen supersolution inequality, $f(t_2)\geq g(t_2)$ is automatically satisfied, making these match as well.

\emph{Case 4}: Let now $p,q$ be causally unrelated and $\bar{p}\leq \bar{q}$. In any curvature bound, $f(t_2)$ is not defined, so this does not correspond to a Jensen subsolution (resp.\ supersolution) inequality. 
%We thus have to show that triangle comparison holds in this case. 
Note that triangle comparison from below is automatically satisfied. 
Triangle comparison from above does not hold if and only if $\bar{p}\ll \bar{q}$, so we restrict to that case, seeking a contradiction. 
As $\bp\leq\bq\ll\by$, we also infer that $p\ll y$. 
Consider $q_t=\gamma(t)$ and $\bar{q}_t=\bar{\gamma}(t)$ for $t>t_2$, then $\bar{p}\ll \bar{q}_t$. 
Note that for large enough $t<t_3$, we do have $p\leq q_t$. 
Set $t'$ to be the infimum of these $t$, then 
\begin{equation}
\lim_{t\searrow t'}\tau(p,q_t) = 0 < \tau(\bar{p},\bar{q}) \leq \tau(\bar{p},\bar{q}_{t'}) = \lim_{t\searrow t'}\tau(\bar{p},\bar{q}_t) \, ,    
\end{equation}
where both limits exist since their arguments are monotonically increasing in $t$. If $\lim_{t\searrow t'}\tau(p,q_t)$ were positive we would have $p \ll q_t'$, contradicting the fact that $t'$ was an infimum. 
This in turn contradicts the Jensen supersolution inequality for $t$ close enough to $t'$, concluding this case. 

\emph{Case 5}: Let now $p,q$ be causally unrelated and $\bar{p},\bar{q}$ causally unrelated too. Then $f(t_2)$ is undefined, so this does not correspond to a Jensen subsolution (resp.\ supersolution) inequality. For triangle comparison, $\tau(p,q)=\tau(\bar{p},\bar{q})$, so this instance of triangle comparison below (resp.\ above) is automatically satisfied.

\emph{Case 6}: Let $p\leq q$ and $\bar{p}\geq\bar{q}$, i.e.\ the two sides of the Jensen inequality are in different settings with respect to our case distinction. 
As in Case 4, we know that $p\leq q\ll y$ and $\bar{x}\ll\bar{q}\leq\bar{p}$, so $x\ll p\ll y$. 
We claim that we can find a parameter $\tilde{t}_2$ suitable for Case 1 violating the required triangle comparison, thus Case 1 shows that the corresponding Jensen inequality does not hold.  
In other words, under the assumptions of curvature bounds (in the sense of triangle comparison or convexity/concavity), this case cannot occur. 
Set $q_t=\gamma(t)$, then the corresponding point is $\bar{q}_t=\bar{\gamma}(t)$. 
For curvature bounds below, we define the functions $a(t)=\tau(p,\gamma(t))$ and $\bar{a}(t)=\tau(\bar{p},\bar{\gamma}(t))$. 
Then we know that both $a$ and $\bar{a}$ are increasing, they both attain $0$ and some positive value, and for $\varepsilon>0$ small enough, we have that $a(t_2+\varepsilon)>0$ but $\bar{a}(t_2+\varepsilon)=0$ (if the second was positive for all $\varepsilon>0$, we would have $\bar{p}\leq\bar{q}$ as well, so we can apply Case 1 directly). 
Let $\varepsilon' \coloneqq \sup \{ \varepsilon \mid \bar{a}(t_2 + \varepsilon)=0\}$. Then $a(t_2 + \varepsilon') > 0, \bar{a}(t_2 + \varepsilon')=0$. 
This allows us to find a value $\tilde{t}_2$ such that $a(\tilde{t}_2)>\bar{a}(\tilde{t}_2)>0$. 
This violates triangle comparison from below and makes $(t_1,\tilde{t}_2,t_3)$ parameters suitable for Case $1$, thus also the Jensen subsolution inequality fails. 

For curvature bounds above the argument is analogous, using the functions $a(t)=\tau(\gamma(t),p)$ and $\bar{a}(t)=\tau(\bar{\gamma}(t),\bar{p})$ instead.

\emph{Case 7}: The case of $p \geq q$ as well as the case of $p$ and $q$ being causally unrelated and $\bp \geq \bq$ are symmetric with respect to time-orientation.
\end{proof}

\section{Equivalences among curvature bounds}
Here we collect the various interdependencies between synthetic sectional curvature bounds that have been established in the previous sections into the following main result:

\begin{thm}[Equivalent notions of curvature bounds for Lorentzian pre-length spaces]
\label{thm:allEquiv}
Let $X$ be a chronological\footnote{We assume all spaces to be chronological from the onset, but wanted to emphasize this in the main theorem.} \LpLSn. 
Recall that $X$ may have curvature bounded below (resp.\ above) by $K$ in any of the following senses: 
\begin{enumerate}
    \item \label{Timelike triangle comparison} Timelike triangle comparison 
    \item \label{One-sided triangle comparison} One-sided timelike triangle comparison
    \item \label{Causal triangle comparison} Causal triangle comparison
    \item \label{One-sided causal triangle comparison} One-sided causal triangle comparison
    \item \label{Strict causal triangle comparison} Strict causal triangle comparison
    \item \label{One-sided strict causal triangle comparison} One-sided strict causal triangle comparison
    \item \label{Monotonicity comparison} Monotonicity comparison
    \item \label{One-sided monotonicity comparison} One-sided monotonicity comparison
    \item \label{Angle comparison} Angle comparison 
    \item \label{Hinge comparison} Hinge comparison 
    \item \label{Timelike four point condition} Timelike four point condition 
    \item \label{Angle version of timelike four point condition} Angle version of timelike four point condition 
    \item \label{Causal four point condition} Causal four point condition 
    \item \label{Strict causal four-point condition} Strict causal four-point condition 
    \item \label{tau-convexity (resp. concavity) condition} $\tau$-convexity (resp.\ $\tau$-concavity) condition 
\end{enumerate} 
In general, the following relations between these curvature bounds hold: 
\begin{equation*}
\ref{Timelike triangle comparison} \Leftrightarrow 
\ref{One-sided triangle comparison} \Leftrightarrow 
\ref{One-sided causal triangle comparison} \Leftrightarrow 
\ref{Causal triangle comparison} \Leftrightarrow 
\ref{Strict causal triangle comparison} \Leftrightarrow
\ref{One-sided strict causal triangle comparison} \text{, and } 
\ref{Timelike four point condition} \Leftrightarrow \ref{Angle version of timelike four point condition} \, ,    
\end{equation*} 
where in the case of upper curvature bounds for $\ref{Causal triangle comparison} \Leftrightarrow \ref{Strict causal triangle comparison}\Leftrightarrow\ref{One-sided strict causal triangle comparison}$ one needs to additionally assume that $X$ is strongly causal and locally causally closed. 

If $X$ is regular, we additionally have: 
\begin{equation*}
\ref{tau-convexity (resp. concavity) condition} 
\Leftrightarrow \ref{Timelike triangle comparison} \Leftrightarrow  \ref{One-sided monotonicity comparison}\Leftrightarrow \ref{Monotonicity comparison} \Rightarrow \ref{Angle comparison} \Leftrightarrow \ref{Hinge comparison} \Rightarrow \ref{Timelike four point condition} \, ,
\end{equation*}
where in the case of lower curvature bounds for $\ref{Monotonicity comparison} \Rightarrow \ref{Angle comparison}$ one needs to additionally assume that $X$ satisfies \eqref{eq: triangle inequality of angles lower curvature bounds}. 

Finally, if $X$ is strongly causal, regular, and locally $D_K$-geodesic, and in the case of upper curvature bounds additionally is locally causally closed and in the case of lower curvature bounds satisfies \eqref{eq: triangle inequality of angles lower curvature bounds}, then all aforementioned notions of curvature bounds are equivalent. 
All relations between the curvature bounds are depicted Figure \ref{fig: main theorem}. 
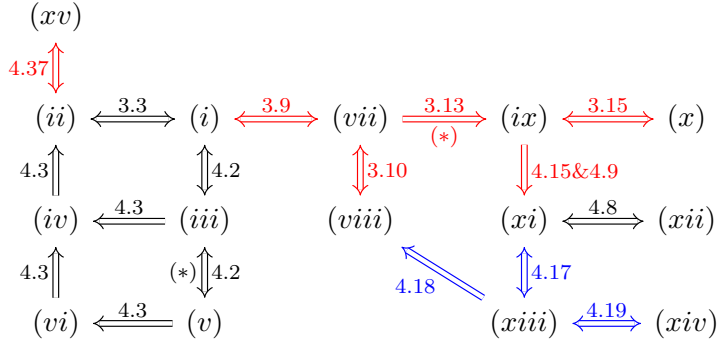
\begin{figure}
\begin{center}
\begin{tikzcd}%[column sep = small]

\ref{tau-convexity (resp. concavity) condition}\arrow[d, Leftrightarrow, color=red, "\ref{pop: triangle comparison and tau convex are equivalent}"'] 
& 
& \\

\ref{One-sided triangle comparison} \arrow[r, Leftrightarrow, "\ref{rem: one-sided triangle comparison and triangle comparison are equivalent}"]& \ref{Timelike triangle comparison}
\arrow[r, Leftrightarrow, color=red, "\ref{pop-eqiv-mon-tr}"] 
\arrow[d, Leftrightarrow, "\ref{thm:equivCausTl}"] 
& \ref{Monotonicity comparison} 
\arrow[r, Rightarrow, color=red, "\ref{pop-implication-mon-->angle}", "(*)"'] 
\arrow[d, Leftrightarrow, color=red, "\ref{rem-cb-one-sided-mon}"]
& \ref{Angle comparison} 
\arrow[r, Leftrightarrow, color=red, "\ref{prop:equivalence_hinge_angle}"]
& \ref{Hinge comparison} \\

\ref{One-sided causal triangle comparison}\arrow[r, Leftarrow,  "\ref{rem: one sided strict causal}"]\arrow[u, Rightarrow,  "\ref{rem: one sided strict causal}"] & \ref{Causal triangle comparison}
\arrow[d, Leftrightarrow, "\ref{thm:equivCausTl}","(*)"'] 
& \ref{One-sided monotonicity comparison} 
\arrow[rd, Leftarrow, color=blue, "\ref{pop: four point condition implies monotonicity comparison}"'] 
& \ref{Timelike four point condition} 
\arrow[u, Leftarrow, color=red, "\ref{pop: angle comparison implies four-point condition CBA} \& \ref{pop: angle comparison implies four-point condition CBB}"'] 
\arrow[r, Leftrightarrow, "\ref{lem-angle version 4pt}"] 
& \ref{Angle version of timelike four point condition} \\

\ref{One-sided strict causal triangle comparison}\arrow[r, Leftarrow,  "\ref{rem: one sided strict causal}"]\arrow[u, Rightarrow,  "\ref{rem: one sided strict causal}"] 
& \ref{Strict causal triangle comparison} 
& & \ref{Causal four point condition} 
\arrow[u, Leftrightarrow, color=blue, "\ref{pop: timelike vs causal 4pt condition}"'] 
\arrow[r, Leftrightarrow, color=blue, "\ref{prop:4-point_causal_noncausal_equivalent}"] 
& \ref{Strict causal four-point condition}
\end{tikzcd}
\end{center}
\caption{All relations between different formulations of curvature bounds for \LpLSn. Black arrows are always valid, red arrows require $X$ to be regular, and blue arrows require $X$ to be strongly causal, regular and locally $D_K$-geodesic. The two instances of additional assumptions for one direction of curvature bounds are marked by $(*)$ in the figure.}
\label{fig: main theorem}
\end{figure}
\end{thm}

\section{Implications of curvature bounds}
In this final section, we prove two implications of upper curvature bounds in the spirit of \cite[Proposition II.2.2, Exercise II.2.3]{BH99}. 
On the one hand, we infer that $\tau$ is bi-concave for curvature bounded above by $0$, and on the other hand we get that $\tau(p, \cdot)$ (resp.\ $\tau(\cdot, p)$) is concave for arbitrary upper curvature bounds. 

\subsection{Bi-concavity of \texorpdfstring{$\tau$}{tau}}
The bi-concavity of $\tau$ for spaces with curvature bounded above by $0$ essentially follows by the intercept theorem in the Minkowski plane. 
\begin{pop}[Bi-concavity of $\tau$]
\label{pop: biconcave}
    Let $X$ be a \LpLS with curvature bounded above by $0$ (in the sense of strict causal triangle comparison). 
    Let $U$ be a comparison neighbourhood in $X$. 
    Then $\tau|_{(U \times U)}$ is ``timelike bi-concave'', i.e., for any two constant speed parametrized timelike distance realizers $\alpha, \beta : [0,1] \to X$ with the same time orientation such that $\alpha(1)\leq\beta(1)$ and $\alpha(0)\leq\beta(0)$ we have 
    \begin{equation}
    \label{eq: bi concavity}
        \tau(\alpha(t),\beta(t)) \geq t \tau(\alpha(1),\beta(1)) + (1-t)\tau(\alpha(0), \beta(0)) \, .
    \end{equation} 
\end{pop}
\begin{proof}
    Say without loss of generality that both curves are future-directed.
    Let us first assume that $\alpha(0)=\beta(0)$. Let $\Delta(\bar{\alpha}(0),\bar{\alpha}(1), \bar{\beta}(1))$ be a comparison triangle for $\Delta(\alpha(0),\alpha(1),\beta(1))$. 
    The intercept theorem yields the elementary equality $\tau(\bar{\alpha}(t),\bar{\beta}(t))=t\tau(\balp(1), \bbet(1))=t\tau(\alpha(1),\beta(1))$. 
    By curvature bounds from above, we infer $\tau(\alpha(t),\beta(t)) \geq \tau(\balp(t), \bbet(t))$, and hence $\tau(\alpha(t),\beta(t)) \geq t\tau(\alpha(1),\beta(1))$, which already shows (\ref{eq: bi concavity}), as $\alpha(0)=\beta(0)$ and $\alpha(t)\leq\beta(t)$. 

    Now assume $\alpha(0) \leq \beta(0)$. Then from $\alpha(0) \leq \beta(0) \ll \beta(1)$, we infer $\alpha(0) \ll \beta(1)$ by the transitivity of $\ll$. 
    Let $\gamma:[0,1] \to X$ be a constant speed parametrized timelike distance realizer from $\alpha(0)$ to $\beta(1)$, i.e., $\gamma(0)=\alpha(0)$ and $\gamma(1)=\beta(1)$. 
    Then $\Delta(\gamma(0), \beta(0), \gamma(1))$ and $\Delta(\gamma(0), \alpha(1), \gamma(1))$
    form two triangles which fit into the special case above (where the first configuration has a reversed time orientation). 
    Thus, we infer $\tau(\alpha(t), \gamma(t)) \geq t\tau(\alpha(1), \gamma(1))$ and $\tau(\gamma(t), \beta(t)) \geq (1-t)\tau(\gamma(0), \beta(0))$. 
    In particular, the intercept theorem also yields that the causal relation in the Minkowski plane is preserved, i.e., we have $\bar{\alpha}(t) \leq \bar{\gamma}(t) \leq \bar{\beta}(t)$ for all $t \in [0,1]$ and hence $\alpha(t) \leq \gamma(t) \leq \beta(t)$ by strict causal curvature bounds.  
    Then via the reverse triangle inequality obtain 
    \begin{align*}
        \tau(\alpha(t), \beta(t)) & \geq 
        \tau(\alpha(t),\gamma(t)) + \tau(\gamma(t), \beta(t)) \\ 
        & \geq t\tau(\alpha(1), \gamma(1)) + (1-t)\tau(\gamma(0), \beta(0)) \\ 
        & = t\tau(\alpha(1), \beta(1)) + (1-t)\tau(\alpha(0), \beta(0)) \, .
    \end{align*}
\end{proof}

\subsection{Concavity of \texorpdfstring{$\tau$}{tau}}
\begin{cor}[Concavity of $\tau$]
Let $X$ be a \LpLS with curvature bounded above by $K$ in the sense of strict causal triangle comparison. Let $U$ be a comparison neighbourhood in $X$. Then $\tau$ is concave on $(U \times U) \cap \tau^{-1}([0,\frac{D_K}{2}))$, i.e., for all $p \in U$ and all timelike geodesics $\gamma$ contained in either $I^{+}(p) \cap U$ or $I^-(p) \cap U$ that have $\tau$-distance less than $\frac{D_K}{2}$ to $p$ we have 
\begin{equation}
\label{eq: concave1}
        \tau(p,\gamma(t)) \geq t \tau(p,\gamma(1)) + (1-t)\tau(p, \gamma(0))
    \end{equation}
    or
    \begin{equation}
    \label{eq: concave2}
        \tau(\gamma(t),p) \geq t \tau(\gamma(1),p) + (1-t)\tau(\gamma(0),p)
    \end{equation}
    for all $t \in [0,1]$. 
\end{cor}
\begin{proof}
An elementary calculation yields that if $\tau$ is bi-concave in the sense of Proposition \ref{pop: biconcave}, then it is also concave in the sense of (\ref{eq: concave1}) and (\ref{eq: concave2}). 
In particular, the current proposition is valid for $K=0$. Moreover, any space with timelike curvature bounded above by $K>0$ also has timelike curvature bounded above by $0$, which can easily be seen from \cite[Lemma 6.1]{AGKS19}. 

Thus, it is only left to consider the case of $K<0$, which up to scaling we can reduce to $K=-1$. 
In this case, we claim that the time separation on $\lm{-1}$, which in the remainder of this proof shall be denoted by $\btau$, is concave. 
Assuming the claim, we then consider a triangle $\Delta(p, \gamma(0), \gamma(1))$ and the corresponding comparison triangle and compute 
\begin{align*}
    \tau(p, \gamma(t)) \geq & \btau(\bp, \bar{\gamma}(t)) \geq t \btau(\bp, \bgam(1)) + (1-t)\btau(\bp, \bgam(0)) \\ 
    = & t \tau(p, \gamma(1)) + (1-t)\tau(p, \gamma(0)) \, .
\end{align*}
Showing that $\btau$ is concave is an elementary calculation. 
Indeed, after applying a suitable Lorentz transformation and parameter shift, we can assume $p=(\cosh(\omega), 0, \sinh(\omega))$ and $\gamma(t)=(\cos(t), \sin(t), 0)$, so that 
\begin{equation}
    f(t) \coloneqq \btau(p,\gamma(t))=\arccos(\cosh(\omega)\cos(t)) \, ,
\end{equation}
which is defined on $(-\pi,-\arccos(\frac{1}{\cosh(\omega)})] \cup [\arccos(\frac{1}{\cosh(\omega)}),\pi)$.  
Then we get
\begin{equation}
    f''(t)=-\frac{\cosh(\omega) \cos (t) (\cosh(\omega)^2 -1)}{\sqrt{1-\cosh(\omega)^2 \cos (t)^2}^3} \, .
\end{equation}
Clearly, both the numerator and the denominator are positive (for $t \leq \frac{\pi}{2}$ and whenever defined), so we have $f''(t) \leq 0$, showing that $\btau(\bp, \bar{\gamma}(t))$ is concave in $t$ on $[-\frac{\pi}{2}, \frac{\pi}{2}]$ (where it is defined). 
\end{proof}
\begin{chapt*}{Acknowledgments}
We want to thank John Harvey and Lewis Napper for helpful discussions. 
We also acknowledge the kind hospitality of the Erwin Schr\"odinger International Institute for Mathematics and Physics (ESI)
during the workshop \emph{Nonregular Spacetime Geometry}, where parts of this research were carried out. 
This work was supported by research grants P33594 and PAT 1996423 of the Austrian Science Fund FWF.
\end{chapt*}
\phantomsection
\addcontentsline{toc}{section}{References}
% \bibliographystyle{alphaabbr}
% \bibliography{references}

\end{document}